\documentclass[]{amsart}
\usepackage[margin=2.5cm]{geometry}
\usepackage{graphicx}
\usepackage{color}
\usepackage{enumerate}
\usepackage{tikz}
\usepackage{float}
\usepackage{graphicx}
\usepackage{subcaption}

\usepackage{comment}

\newtheorem{thm}{Theorem}[section]
\newtheorem{cor}[thm]{Corollary}
\newtheorem{lem}[thm]{Lemma}
\newtheorem{prop}[thm]{Proposition}
\newtheorem{ex}[thm]{Example}
\theoremstyle{definition}
\newtheorem{defn}[thm]{Definition}
\theoremstyle{remark}
\newtheorem{rem}[thm]{Remark}

\numberwithin{equation}{section}

\newcommand{\abs}[1]{\left\vert#1\right\vert}

\newcommand{\eps}{\varepsilon}

\usepackage{etoolbox}
\patchcmd{\section}{\scshape}{\bfseries}{}{}
\makeatletter
\renewcommand{\@secnumfont}{\bfseries}
\makeatother

\newcounter{nameOfYourChoice}

\begin{document}


\title[]{Geometric bounds for the Magnetic Neumann eigenvalues in the plane}%

\author{Bruno Colbois}
\author{Corentin Léna}
\address{}
\author{Luigi Provenzano}
\author{Alessandro Savo}

\address{Bruno Colbois, Université de Neuch\^atel, Institute de Mathémathiques, Rue Emile Argand 11, 2000 Neuch\^atel, Switzerland}
\email{bruno.colbois@unine.ch}
\address{Corentin Léna, Università degli Studi di Padova, Dipartimento di Tecnica e Gestione dei Sistemi Industriali (DTG), Stradella S. Nicola 3, 36100 Vicenza, Italy}
\email{corentin.lena@unipd.it}
\address{Luigi Provenzano, Sapienza Università  di Roma, Dipartimento di Scienze di Base e Applicate per l'Ingegneria (SBAI), Via Antonio Scarpa 16, 00161 Roma, Italy}
\email{luigi.provenzano@uniroma1.it}
\address{Alessandro Savo, Sapienza Università  di Roma, Dipartimento di Scienze di Base e Applicate per l'Ingegneria (SBAI), Via Antonio Scarpa 16, 00161 Roma, Italy}
\email{alessandro.savo@uniroma1.it}

\thanks{The first author acknowledges support of the SNSF project ‘Geometric Spectral Theory’, grant number 200021-19689. The third and fourth author acknowledge support of the Gruppo Nazionale per le Strutture Algebriche, Geometriche e le loro Applicazioni (GNSAGA) of the I\-sti\-tuto Naziona\-le di Alta Matematica (INdAM)}




\date{}%

\begin{abstract}
We consider the eigenvalues of the magnetic Laplacian on a bounded domain $\Omega$  of $\mathbb R^2$ with uniform magnetic field $\beta>0$ and magnetic Neumann boundary conditions. We find upper and lower bounds for the ground state energy $\lambda_1$ and we provide semiclassical estimates in the spirit of Kr\"oger for the first Riesz mean of the eigenvalues. We also discuss upper bounds for the first eigenvalue for non-constant magnetic fields $\beta=\beta(x)$ on a simply connected domain in a Riemannian surface.\\
In particular: we prove the upper bound $\lambda_1<\beta$ for a general plane domain, and the upper bound $\lambda_1<\sup_{x\in\Omega}\abs{\beta(x)}$ for a variable magnetic field when $\Omega$ is simply connected. \\
For smooth domains, we prove a lower bound of $\lambda_1$ depending only on the intensity of the magnetic field $\beta$ and the rolling radius of the domain.
\\The estimates on the Riesz mean imply an upper bound for the averages of the first $k$ eigenvalues which is sharp when $k\to\infty$ and consists of the semiclassical limit $\dfrac{2\pi k}{\abs{\Omega}}$ plus an oscillating term. \\
We also construct several examples, showing the importance of the topology: in particular we show that an arbitrarily small tubular neighborhood of a generic simple closed curve has lowest eigenvalue bounded away from zero, contrary to the case of a simply connected domain of small area, for which $\lambda_1$ is always small. 
\end{abstract}

\maketitle

\tableofcontents





\section{Introduction}

The main scope of this paper is to derive upper and lower bounds for the eigenvalues of the magnetic Laplacian with constant magnetic field $\beta>0$ and magnetic Neumann boundary conditions on domains of $\mathbb R^2$. Specifically, we will consider the magnetic Laplacian associated with the potential $1$-form 
\begin{equation}\label{standard}
A=\frac{\beta}{2}(-x_2dx_1+x_1dx_2)
\end{equation}
which  generates the magnetic field of constant strength $\beta$, in the sense that $dA=\beta dv$ where $dv$ is the volume 2-form. Note that replacing $\beta$ by $-\beta$ does not change the spectrum, therefore it is not restrictive to consider $\beta>0$, see Subsection \ref{properties}. The eigenvalues correspond to the energy levels of a quantum charged particle in a two-dimensional region subject to a transversal magnetic field of constant strength $\beta$. It is clear that the interest in the study of the corresponding spectrum  originates in Quantum Mechanics and Mathematical Physics. We refer to the books \cite{FH_book,ray} for more detailed discussions on the topic. 

\smallskip

Nevertheless the subject has attracted a lot of attention in the last decades also in Analysis and Geometry. Relevant questions which are usually posed in these contexts include geometric bounds for the eigenvalues and isoperimetric inequalities. In the present paper we will focus on eigenvalue bounds, and in particular on how the geometry of the domain influences the eigenvalues, with particular attention to the ground state energy $\lambda_1(\Omega,\beta)$, which turns out to be positive. Concerning previous results on eigenvalue bounds for the magnetic Neumann problem, we refer to \cite{CEIS,CS1,ELMP,kovarik,FH2,mag_cheeger,LS}.

\smallskip

 In this paper the notation $\lambda_j(\Omega,\beta)$ refers to the $j$-th eigenvalue of the magnetic Laplacian with Neumann conditions and potential $A$ as in \eqref{standard}
(When $\Omega$ is not simply connected, the choice of the potential form generating the magnetic field $\beta$ may affect the spectrum, see Subsection \ref{cmf}.)

\smallskip

If we set $\beta=0$ in \eqref{standard}, that is, if $A=0$, we fall back into the case of the Neumann Laplacian, for which a huge literature on eigenvalue bounds is available. In our notation, $\lambda_1(\Omega,0)=0$, while $\lambda_2(\Omega,0)>0$ is the first positive eigenvalue of the (non-magnetic) Neumann Laplacian. On the other hand, when $\beta>0$, the magnetic spectrum, in particular $\lambda_1(\Omega,\beta)$, displays a peculiar behavior when compared  to the usual Laplacian spectrum. We will list here just a few instances in order to give a glimpse of this fact. 

$ - $ First, the behavior of the first eigenvalue under homotheties involves the strength of the magnetic field: for any $\alpha>0$ one has:
$$
\lambda_j(\alpha\Omega,\beta)=\dfrac{1}{\alpha^2}\lambda_j(\Omega,\beta\alpha^2).
$$

$-$ This and the upper bound \eqref{A0} imply that $\lambda_1(\alpha\Omega,\beta)\to 0$ as $\alpha\to 0^+$: the first eigenvalue vanishes when the domain is homothetically shrunk.

$-$ However, the first eigenvalue does not necessarily go to $0$ when $|\Omega|\to 0^+$: there exist domains with arbitrarily small area and first eigenvalue bounded away from zero (as a matter of fact, a small tubular neighborhood of a ``generic'' simple closed curve has first eigenvalue bounded away from zero, see Example \ref{annuli0}).

$-$ Still concerning homotheties, given any domain $\Omega$, we have $\lambda_1(\alpha\Omega,\beta)\to\Theta_0\beta$ as $\alpha\to +\infty$, where $\Theta_0\approx 0.590106$ is a universal constant (de Gennes constant, see \cite[Chapter 3]{FH_book}).  This implies that an arbitrarily large volume does not imply a small first eigenvalue. Moreover, note that the function $\alpha\mapsto\lambda_1(\alpha\Omega,\beta)$ is not generally increasing (see Figure \ref{F1} when $\Omega$ is a disk).

$-$ There exist convex domains with inradius bounded below by a positive constant and first eigenvalue arbitrarily small (see Example \ref{ex:tr}).

$-$ There are striking differences between the magnetic Neumann and the magnetic Dirichlet eigenvalues. 

Let us briefly comment on that. Let  $\lambda_1^D(B_R,\beta)$ denote the first magnetic Dirichlet eigenvalue on a disk of radius $R$. It is quite standard to prove that  $\lambda_1^D(B_R,\beta)$ is decreasing from $+\infty$ to $\beta$ (which is a strict and sharp lower bound) as a function of $R\in(0,+\infty)$, and that the first eigenfunction is real and radial for any $R$ (see e.g., \cite{son}). Moreover, the Faber-Krahn inequality holds for $\lambda_1^D(\Omega,\beta)$, see \cite{erdos_FK}. On the other hand, the understanding of the behavior of the first Neumann eigenvalue $\lambda_1(B_R,\beta)$ on disks as a function of $R$ is very complicated: the first eigenfunction has angular momentum which increases with $R$; the eigenvalue is uniformly bounded, vanishes as $R\to 0^+$ and presents an oscillating (i.e., non-monotonic) behavior as a function of $R$; from numerical studies it seems that $\lambda_1(B_R,\beta)<\Theta_0\beta$ for all $R$, but we have no proof of this fact at the moment. See Figure \ref{DvsN} for a plot of $\lambda_1^D(B_R,1)$ and $\lambda_1(B_R,1)$ as functions of $R$. We refer to Appendix \ref{sec:disk} for more details on the Neumann problem for disks.

\begin{figure}
\includegraphics[width=0.7\textwidth]{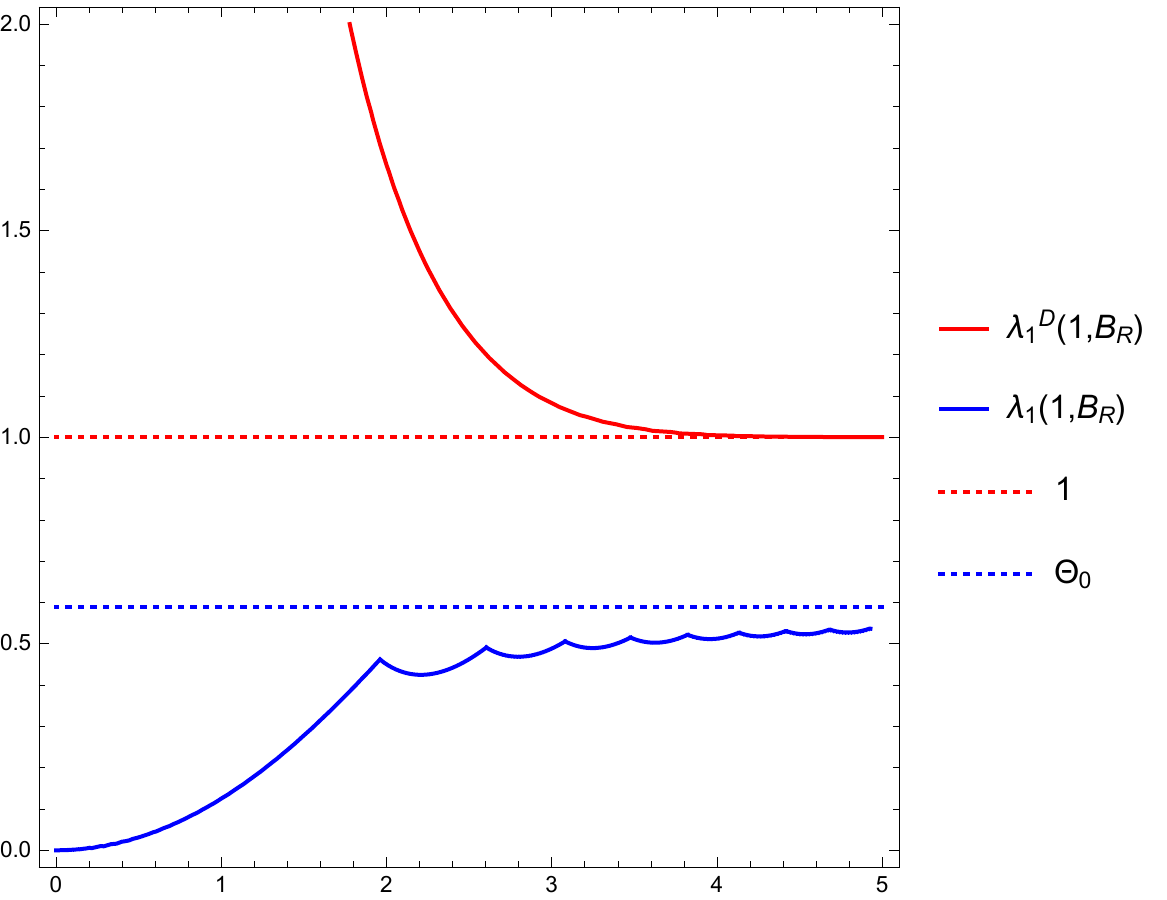}
\caption{First magnetic Dirichlet eigenvalue (in red) of a disk $B_R$ as a function of $R$; first magnetic Neumann eigenvalue (in blue) of a disk $B_R$ as a function of $R$. Here $\beta=1$.}
\label{DvsN}
\end{figure}

$ - $ Finally, we note that the reverse Faber-Krahn inequality for the first magnetic Neumann eigenvalue is still an open problem, and it definitely does not hold for multiply connected domains. In fact, in \cite{FH2} the authors show that, given an annulus $\Omega$, there exists $\beta_0(\Omega)$ such that, for any $\beta>\beta_0$ $\lambda_1(\Omega^*,\beta)<\lambda_1(\Omega,\beta)$, where $\Omega^*$ is a disk with $|\Omega^*|=|\Omega|$. Note that this is an asymptotic counterexample. More easily, from our Example \ref{annuli0} we see that, for any $\beta>0$, there exist plenty of non-simply connected domains which have larger first eigenvalue than the disk of the same volume. In fact, as already mentioned, a small tubular neighborhood (of small area) of a ``generic'' simple closed curve has first eigenvalue uniformly bounded away from zero, while a disk with the same (small) area has small first eigenvalue.

\smallskip

These few examples show that understanding which geometrical properties imply upper or lower bounds on $\lambda_1(\Omega,\beta)$ is not trivial. 

\smallskip

In the present paper we improve the known bounds in different ways.
First, we focus on the ground state energy $\lambda_1(\Omega,\beta)$, which is strictly positive for any value of $\beta>0$. We prove a universal upper bound, valid for any domain, which is strict and given by $\beta$, the intensity of the magnetic field (Theorem \ref{TA}). For certain classes of domains, which include sub-graphs and self-tiling domains, we prove that an upper bound is given by $\Theta_0\beta$ (Theorem \ref{TB}), which is optimal in view of the asymptotic behavior $\lambda_1(\Omega,\beta)=\Theta_0\beta+o(\beta)$ as $\beta\to +\infty$ (see \cite[Chapter 5]{FH_book}). We also prove a general upper bound for the first eigenvalue when the magnetic field is non constant, on simply connected Riemannian surfaces, in terms of the sup-norm of the magnetic field (see Theorem \ref{TF} and also Section \ref{sub:variable} for a discussion on the non simply connected case). 

We continue by considering lower bounds for $\lambda_1(\Omega,\beta)$. Our starting point are the lower bounds proved in \cite{kovarik} for simply connected domains, which we use to produce a new lower bound for arbitrary smooth domains in terms of $\beta$ or $\beta^2$, depending only on the rolling radius $\delta$ of the domain, see Theorem \ref{TC}.

We then consider the whole spectrum and we prove semiclassical estimates on eigenvalues averages in the spirit of Kr\"oger \cite{Kro}, which are asymptotically sharp (Theorem \ref{TD}). These estimates imply upper bounds on single eigenvalues of any order. Note that bounds on eigenvalue averages turn out to be equivalent to bounds on the first Riesz mean $R_1(z)$, in the tradition of Berezin and Li-Yau \cite{Berezin,LiYau}. Lower bounds for Riesz-means in case of variable magnetic field have been also obtained in \cite{CEIS}. Upper bounds for Riesz means already exist for the Dirichlet magnetic eigenvalues in \cite {erdos} (see also \cite{frank_magnetic}). We note that the behavior of our lower bounds on $R_1$ is significantly different from the behavior of the upper bounds for Dirichlet Riesz means \cite{erdos}, and reflects the interplay between the area of the domain, the strength of the magnetic field and the eigenvalue index.

We include three appendices, where we discuss results which are related to eigenvalue bounds, but have an interest on their own. Namely, we consider the magnetic Laplacian on embedded curves, establishing in this setting a sort of ``reverse Faber-Krahn'' inequality. We also discuss the case of disks and collect a few other examples, which are instructive in order to understand some of the difficulties in establishing precise bounds.

\smallskip

The  paper is organized as follows. In Section \ref{sec:results} we introduce the mathematical problem, fix the notation, and state our main results. In Sections \ref{sec:up} and \ref{sec:lo} we prove, respectively, the upper and the lower bounds for $\lambda_1(\Omega,\beta)$. In Section \ref{sub:variable} we prove upper bounds for the first eigenvalue in the case of variable magnetic field on a Riemannian surface. In Section \ref{sec:sums} we prove the asymptotically sharp, semiclassical estimates for Riesz means  and averages. In Appendix \ref{sec:thin} we study the eigenvalue problem obtained by restricting the magnetic potential to embedded curves, and prove an isoperimetric result. In Appendix \ref{sec:disk} we collect a few properties of magnetic eigenvalues on disks. In Appendix \ref{sub:examples} we provide further examples which help to clarify the difficulties in finding good bounds.

\section{Notation and statement of results}\label{sec:results}

\subsection{Generalities on the magnetic Laplacian}

 Let $\Omega$ be a bounded domain in $\mathbb R^2$ and let $A$ be a smooth real $1$-form.  We define the {\it magnetic differential} of a smooth complex valued  function $u$ as the complex $1$-form defined as follows:
$$
d^Au=du-iuA.
$$
The adjoint of $d^A$ is the operator $\delta^A$ acting on a $1$-form $\omega$ as $\delta^A\omega=\delta\omega+i\omega(A^{\sharp})$, where $A^{\sharp}$ is the dual vector field of $A$ and $\delta$ is the adjoint of $d$ (note that $\delta=-{\rm div}$). The {\it magnetic Laplacian} associated to the potential $A$ is then defined as
$
\Delta_Au=\delta^Ad^Au.
$
A standard calculation shows that 
$$
\Delta_Au=\Delta u+|A|^2u+2i\langle du, A \rangle+iu\,{\rm div}\,A,
$$
where ${\rm div A}=-\delta A$ is the usual divergence of the $1$-form $A$. 
Note that in this paper $\Delta$ denote the positive-definite operator
$
\Delta:=-\partial^2_{x_1x_1}-\partial^2_{x_2x_2}.
$ 
By $x$ we denote a point in $\mathbb R^2$ with Cartesian coordinates $(x_1,x_2)$. We will often use polar coordinates $(r,t)\in[0,+\infty)\times[0,2\pi]$. In particular $r=|x|=\sqrt{x_1^2+x_2^2}$.

\medskip

We will often identify, by abuse of language (and when this will not create confusion) the form $A$ with its dual vector field $A^{\sharp}$ ({\it vector potential}); the magnetic Laplacian can then be written
in the following form, often found in the literature:
$$
\Delta_A=-(\nabla-iA)^2.
$$
Dually, we define the {\it magnetic gradient} of a complex function $u$ as
$$
\nabla^Au:= \nabla u-iuA
$$
where $A$ is thought as a vector potential.
We consider the eigenvalue problem for the magnetic Laplacian with {\it magnetic Neumann conditions} in $\Omega$, namely
\begin{equation}\label{AB_N_0}
\begin{cases}
\Delta_A u=\lambda u\,, & {\rm in\ } \Omega,\\
\langle\nabla^Au,N\rangle=0\,, & {\rm on\ }\partial\Omega,
\end{cases}
\end{equation}
where $N$ is the outer unit normal to $\partial\Omega$. In the second line, one can see $\langle\nabla^Au,N\rangle$ as being the magnetic normal derivative of $u$.

\smallskip

Problem \eqref{AB_N_0} is understood in the weak sense as follows: find $u\in H^1(\Omega)$ and $\lambda\in\mathbb R$ such that
\begin{equation}\label{AB_N_0_weak}
\int_{\Omega}\langle\nabla^Au,\overline{\nabla^A\phi}\rangle=\lambda\int_{\Omega} u\overline\phi\,,\ \ \ \forall\phi\in H^1(\Omega).
\end{equation}
Here $H^1(\Omega)$ is the standard Sobolev space of complex valued functions in $L^2(\Omega)$ with weak first derivatives in $L^2(\Omega)$.

It is standard to prove that, under reasonable assumptions on $\Omega$ (e.g., $\Omega$ Lipschitz) problem \eqref{AB_N_0_weak} admits an increasing sequence of positive eigenvalues of finite multiplicity diverging to $+\infty$
$$
0<\lambda_1\leq\lambda_2\leq\cdots\leq\lambda_j\leq\cdots\nearrow+\infty.
$$
Through the rest of the paper we shall implicitly assume that $\Omega$ is a bounded domain for which the spectrum of \eqref{AB_N_0_weak} is discrete. The eigenvalues are variationally characterized as follows
\begin{equation}\label{minmax}
\lambda_j=\min_{\substack{U\in H^1(\Omega)\\{\rm dim}\,U=j}}\max_{0\ne u\in U}\frac{\int_{\Omega}|\nabla^Au|^2}{\int_{\Omega}|u|^2}.
\end{equation}
Note that $\lambda_j=\lambda_j(\Omega,A)$ normally depends on both the domain and the potential $1$-form; however we observe the well-known {\it gauge invariance} of the spectrum, according to which if we replace $A$ by $A+df$, for any smooth function $f$, the spectrum remains unchanged:
$$
\lambda_j(\Omega,A)=\lambda_j(\Omega,A+df).
$$

\subsection{Constant magnetic field}\label{cmf} In the present paper we will mainly consider the following potential $1$-form $A$ in \eqref{AB_N_0}: 
\begin{equation}\label{magnetic_potential}
A=\frac{\beta}{2} (-x_2dx_1+x_1dx_2), 
\end{equation}
which we will often call {\it standard potential};  here $\beta$ is a positive constant. Note that $dA=\beta dv$ is a constant magnetic field of strength $\beta$.  Also observe that ${\rm div}\,A=0$. Here $dv$ is the usual area element in $\mathbb R^2$.

When it is necessary for the purposes of the presentation, we shall highlight the dependence of the eigenvalues $\lambda_j$ on  $\beta$ and $\Omega$, and denote them as 
\begin{equation}\label{notation}
\lambda_j(\Omega,\beta). 
\end{equation}
In other words, the notation in \eqref{notation} refers to the eigenvalues of the magnetic Laplacian with potential form as in \eqref{magnetic_potential}. 

\smallskip

To this regard, a remark is perhaps in order here. Observe that if $\Omega$ is simply connected, then the spectrum of \eqref{AB_N_0} depends only on $\beta$ and not on the potential $A$. In fact, if $A,A'$ are two potentials such that $dA=dA'=\beta dv$, then they differ by a closed $1$-form, which is exact on $\Omega$ provided that $\Omega$ is simply connected: in that case the spectra of \eqref{AB_N_0} with $A$ and $A'$ coincide by gauge invariance.  The situation is completely different if the domain is not simply connected, in which case the spectra corresponding to $A,A'$ differing by a closed $1$-form in general may not coincide. In this case, as we have already declared, we are considering the spectrum of \eqref{AB_N_0} with $A$ defined by \eqref{magnetic_potential}.

\subsection{Upper bounds for $\lambda_1$}
We list here the main results concerning upper bounds on the ground state energy $\lambda_1(\Omega,\beta)$.

The first result is an upper bound for $\lambda_1(\Omega,\beta)$, valid for any bounded domain in $\mathbb R^2$. We present its proof in Section \ref{sec:up} (Theorem \ref{up_B_diam}).

\begin{thm}\label{TA}
Let $\Omega$ be a bounded domain in $\mathbb R^2$. Then
\begin{equation}\label{BA}
\lambda_1(\Omega,\beta)<\beta.
\end{equation}
\end{thm}
Actually, the upper bound \eqref{BA} is a consequence of a more precise bound that we establish in Theorem \ref{up_B_diam}, namely
$$
\lambda_1(\Omega,\beta)\leq
\begin{cases}
\beta-\frac{1}{2R_{\Omega}^2}\,, & {\rm if\ }R_{\Omega}>\frac{1}{\sqrt{\beta}}\,,\\
\frac{R_{\Omega}^2\beta^2}{2}\,, & {\rm if\ }R_{\Omega}\leq\frac{1}{\sqrt{\beta}},
\end{cases}
$$
where $R_{\Omega}$ denotes the circumradius of $\Omega$, namely, the radius of the smaller disk containing $\Omega$. Alternatively, for simply connected domains the upper bound \eqref{BA} follows from
$$
\lambda_1(\Omega,\beta)\leq\beta(1-e^{-\frac{\beta|\Omega|}{2\pi}})
$$
which we prove in Theorem \ref{variable} (see Theorem \ref{TF} here below). Note that this latter bound implies that for simply connected domains, $\lambda_1(\Omega,\beta)\to 0$ as $|\Omega|\to 0^+$. This is no longer true if $\Omega$ is not simply connected, see Example \ref{annuli0}.
\medskip

For certain classes of domains we prove an asymptotically sharp upper bound. This bound depends on a universal positive constant, the {\it de Gennes constant} $\Theta_0\approx 0.590106$, which we discuss in more detail in the next subsection (see \eqref{T00}).  Namely, we prove the following
\begin{thm}\label{TB}
Let $\Omega$ be a bounded domain in $\mathbb R^2$. Assume that, up to isometries, one of the following holds:
\begin{enumerate}[1)]
\item $\Omega$ is a sub-graph, namely $\Omega=\{(x_1,x_2)\in\mathbb R^2:a<x_1<b,0<x_2<g(x_1)\}$ for some smooth $g:(a,b)\to[0,+\infty)$;
\item $\Omega$ is contained in some strip $(a,b)\times(0,+\infty)$ and contains $(a,b)\times(0,2\sqrt{\Theta_0/\beta})$.
\setcounter{nameOfYourChoice}{\value{enumi}}
\end{enumerate}
Then
\begin{equation}\label{BB1}
\lambda_1(\Omega,\beta)< \Theta_0\beta.
\end{equation}
Assume that
\begin{enumerate}[1)]
\setcounter{enumi}{\value{nameOfYourChoice}}
\item $\Omega$ is self-tiling, namely $\Omega={\rm Int}\left(\bigcup_{i=1}^m\overline\omega_i\right)$, where each $\omega_i$ is isometric to $\frac{1}{\sqrt{m}}\Omega$.
\end{enumerate}
Then
\begin{equation}\label{BB2}
\lambda_1(\Omega,\beta)\leq\Theta_0\beta.
\end{equation}
\end{thm}
The constant $\Theta_0$ is related with the asymptotic behavior of $\lambda_1(\Omega,\beta)$ as $\beta\to+\infty$, in fact, for any smooth $\Omega$ it holds $\Theta_0=\lim_{\beta\to+\infty}\lambda_1(\Omega,\beta)/\beta$ (see \eqref{asympt_B}). Hence the bounds \eqref{BB1}-\eqref{BB2} are optimal in this sense. We can pose then the following question.

\medskip

\noindent{\bf Open problem 1.} Prove that $\lambda_1(\Omega,\beta)<\Theta_0\beta$ for all $\Omega$.

\medskip

Note that our Theorem \ref{TB} does not cover the case of all disks (disks of small radius  are covered by Theorems \ref{TA} and \ref{TF}). However, numerical evidences (see Appendix \ref{sec:disk}) suggest that $\Theta_0\beta$ is a strict upper bound for all disks. 

\medskip

\noindent{\bf Open problem 2.} Prove that  $\lambda_1(B,\beta)<\Theta_0\beta$ for any disk $B$.
\medskip

Points 1) and 2) of Theorem \ref{TB} are proved in Theorem \ref{T0_1}, while point 3) is proved in Theorem \ref{propUpper}. In the case 3) we actually prove in Theorem \ref{propUpper} that $\lambda_1(\Omega,\beta)\leq\Lambda(\Omega)\beta$, where $\Lambda(\Omega)\leq\Theta_0$ is given by $\lim_{\beta\to+\infty}\lambda_1(\Omega,\beta)/\beta$. For certain domains with convex corners it is possible to show that $\Lambda(\Omega)<\Theta_0$, see \cite{vir1}. Self-tiling domains with this property are, for example, triangles and parallelograms.

\medskip

Concerning $\lambda_1$ we have also considered the case of a variable magnetic field $\beta=\beta(x)$ on simply connected Riemannian surfaces. We have established a general upper bound which depends only on the field $\beta$, and not on the magnetic potential.

\begin{thm}\label{TF}
Let $\Omega$ be a simply connected, orientable, compact Riemannian surface, let $A$ be a smooth $1$-form, and let $\beta:\Omega\to\mathbb R$ be defined by $dA=\beta dv$. If $\partial\Omega=\emptyset$, assume also $\int_{\Omega}\beta=0$. Let $\lambda_1(\Omega,A)$ denote the first eigenvalue of \eqref{AB_N_0} with magnetic potential $A$. Let $\phi:\Omega\to\mathbb R$ be the unique solution to
$$
\begin{cases}
\Delta	\phi=\beta\,, & {\rm in\ }\Omega,\\
\phi=0\,, & {\rm on\ }\partial\Omega {\rm\ (if\ }\partial\Omega\ne\emptyset{\rm )}.
\end{cases}
$$
and let $\beta^*:=\max_{\overline\Omega}|\beta|$, $\phi^*:=\max_{\overline\Omega}|\phi|$. The following inequalities hold:
\begin{enumerate}[1)]
\item $\lambda_1(\Omega,A)<\beta^*$.
\item If $\beta\geq 0$, then $\lambda_1(\Omega,A)\leq\beta^*(1-e^{-2\phi^*})$.
\item  If $\beta\geq 0$ and $\Omega$ is a domain of $\mathbb R^2$, then $\lambda_1(\Omega,A)\leq\beta^*(1-e^{-\frac{\beta^*|\Omega|}{2\pi}})$.
\item If $\Omega$ is a domain of $\mathbb R^2$ and $A$ is given by the standard potential \eqref{magnetic_potential}, then 
$$
\lambda_1(\Omega,A)=\lambda_1(\Omega,\beta)\leq\beta(1-e^{-\frac{\beta|\Omega|}{2\pi}}).
$$
\end{enumerate}
\end{thm} 
Theorem \ref{TF} is a consequence of Theorem \ref{variable} and Corollary \ref{variable_1}. In  Section \ref{sub:variable} we also discuss bounds in the case of non-simply connected domains.

\begin{rem}\label{rem_not_sc}We remark that 3) cannot hold if $\Omega$ is not simply connected. In fact, as we show in Example \ref{annuli0}, there exist non-simply connected domains $\omega_h$ with $|\omega_h|\to 0$ and $\lambda_1(\omega_h,\beta)\to c>0$ as $h\to 0^+$. On the other hand, if $\Omega$ is simply connected, 3) implies that $\lambda_1(\Omega,A)\to 0$ as $|\Omega|\to 0$. In the case of $\beta>0$ constant, this can be also deduced from \cite[Theorem 1.2]{FH2}, which implies that $\lambda_1(\Omega,\beta)\leq\frac{\beta^2|\Omega|}{8\pi}$.
\end{rem}


\subsection{Lower bounds for $\lambda_1$}
The next result concerns lower bounds for $\lambda_1(\Omega,\beta)$. In order to state the result, we recall that a domain is said to satisfy the $\delta$-interior ball condition with $\delta>0$ if for any $x\in\partial\Omega$ there exists a disk of radius $\delta>0$ tangent to $\partial\Omega$ at $x$ and entirely contained in $\Omega$. In more technical terms, this condition can be also expressed by saying that the injectivity radius of the boundary is bounded below by $\delta$.

\begin{thm}\label{TC}
Let $\Omega$ be a smooth bounded domain satisfying the $\delta$-interior ball condition. Then there exists a universal constant $C>0$ such that
\begin{enumerate}[1)]
\item $\lambda_1(\Omega,\beta)\geq C\beta^2\delta^2$ if $\beta\delta^2\leq 1$;
\item $\lambda_1(\Omega,\beta)\geq C\beta$ if $\beta\delta^2\geq 1$.
\end{enumerate}
\end{thm}
Theorem \ref{TC} is proved in Section \ref{sec:lo} (Theorem \ref{thm_lower}). Its proof relies on a combination of the lower bounds for magnetic eigenvalues in \cite{kovarik} (see Theorem \ref{kov1}) and for the Laplacian eigenvalues in \cite{ChLi1997} (see Theorem \ref{ChenLi}).  Note that the behavior of our lower bounds in $\delta$ and $\beta$ is consistent with  the upper bounds of Theorem \ref{TA}. It is also consistent with the asymptotic behavior of the first eigenvalue with respect to $\beta$ as $\beta\to 0^+$ and $\beta\to+\infty$. We refer to Remark \ref{rem:lowerthm} for more discussions on the sharp behavior in $\beta$ and $\delta$ of the lower bounds of Theorem \ref{TC}. A series of examples show that in many situations the bounds given by Theorem \ref{TC} are good in capturing the behavior of the first eigenvalue (see Examples \ref{ex_kov}, \ref{conv_small}, \ref{open_curve}). In Section \ref{sec:lo}  we also prove lower bounds for star-shaped domains in terms of the radii $0<R<R_0$ of two disks $B(p,R),B(p,R_0)$ such that $B(p,R)\subset\Omega\subset B(p,R_0)$ (Proposition \ref{lo_star}).

\subsection{Upper bounds for higher eigenvalues and averages}
The next results involve upper bounds for all the eigenvalues. By means of the so-called averaged variational principle (Theorem \ref{thm:AVP}) we obtain asymptotically sharp lower bounds on the first Riesz mean $R_1$ of magnetic eigenvalues $\lambda_j$, which is defined by $R_1(z):=\sum_{j=1}^{\infty}(z-\lambda_j)_+$, where $a_+:=\max\{0,a\}$. Here we drop the dependence of $\lambda_j$ on $\beta$ and $\Omega$. Lower bounds on $R_1(z)$ are equivalent to  upper bounds for eigenvalues averages.

\begin{thm}\label{TD}
For all $z\geq 0$ we have
\begin{equation}\label{BD1}
R_1(z)\geq\frac{|\Omega|}{8\pi}z^2-\frac{|\Omega|\beta^2}{2\pi}\psi^2\left(\frac{z}{2\beta}+\frac{1}{2}\right),
\end{equation}
where $\psi(a)=a-[a]-\frac{1}{2}$ denotes the fluctuation function of $a\in\mathbb R$, and $[a]$ denotes the integer part of $a\in\mathbb R$. Moreover, for any $k\in\mathbb N$, $k\geq 1$  we have
\begin{equation}\label{BD2}
\begin{cases}
\frac{1}{k}\sum_{j=1}^k\lambda_j\leq\beta\,, & {\rm if\ }k\leq\frac{\beta|\Omega|}{2\pi}\\
\frac{1}{k}\sum_{j=1}^k\lambda_j\leq\frac{2\pi k}{|\Omega|}+R\left(\frac{2\pi k}{\beta|\Omega|}\right)\,, & {\rm if\ }k>\frac{\beta|\Omega|}{2\pi},
\end{cases}
\end{equation}
where $R(a)=\frac{\beta}{a}(a-[a])([a]-a+1)\in[0,\beta/4a]$.
\end{thm}

Theorem \ref{TD} is proved in Section \ref{sec:sums} (Theorem \ref{main_averages}). We note that our upper bounds are asymptotically sharp, in fact Weyl's law for magnetic eigenvalues implies that $R_1(z)=\frac{|\Omega|}{8\pi}z^2+o(z^2)$ as $z\to+\infty$, or, equivalently, $\frac{1}{k}\sum_{j=1}^k\lambda_j=\frac{2\pi k}{|\Omega|}+o(k)$ as $k\to+\infty$. The bounds given in Theorem \ref{TD} are the analogue of the Kr\"oger upper bounds for the averages of Laplacian eigenvalues \cite{Kro}. Note that the Weyl term $\frac{2\pi k}{|\Omega|}$ appears in the estimates \eqref{BD2} only for large $k$, and this is natural, since magnetic eigenvalues do not scale as Laplacian eigenvalues.  The first inequality of \eqref{BD2} tells us that for small $k$ (depending on $\beta$ and $|\Omega|$) the average of the first $k$ eigenvalues is smaller than $\beta$. This is somehow sharp, as this behavior can be observed in the case of disks, see Appendix \ref{sec:disk}. However, as $k\to+\infty$, the upper bound is given by the semiclassical limit $\frac{2\pi k}{|\Omega|}$, plus a remainder term which is oscillating, bounded, and of $o(1/k)$ as $k\to+\infty$.

As a corollary of Theorem \ref{TD} we get upper bounds on single eigenvalues (see Corollary \ref{single_bounds}).

\begin{cor}\label{CD}
For all $k\in\mathbb N$ we have
\begin{equation}
\lambda_{k+1}(\Omega,\beta)\leq\frac{8\pi k}{|\Omega|}+\beta.
\end{equation}
\end{cor}
Note that this agrees with the fact that the second eigenvalue might go to $+\infty$ as $|\Omega|\to 0^+$ (contrarily to the first eigenvalue).

\subsection{A reverse Faber-Krahn inequality for the first eigenvalue of embedded curves}

In Appendix \ref{sec:thin} we consider a one-dimensional eigenvalue problem related with the magnetic Laplacian. Let $\Gamma$ be a simple closed curve bounding some connected domain $\Omega$. We consider $\lambda_1(\Gamma,\hat A)$, the first eigenvalue of a magnetic operator obtained by restricting the magnetic potential $A$ to $\Gamma$. We call $\hat A$ such restriction: by definition, $\hat A(X)=A(X)$ for all tangent vectors $X$ to $\Gamma$.

 It turns out (see Theorem \ref{thm1d}) that
$$
\lambda_1(\Gamma,\hat A)=\frac{4\pi^2}{|\Gamma|^2}\min_{n\in\mathbb Z}\left(n-\frac{\beta|\Omega|}{2\pi}\right)^2.
$$
Note that $\lambda_1(\Gamma,\hat A)=0$ if and only if $\frac{\beta|\Omega|}{2\pi}\in\mathbb N$. We deduce then the following isoperimetric inequality (see Theorem \ref{reverse_curves})
\begin{thm}\label{TE}
Let $\Omega$ be a bounded, simply connected domain with boundary $\Gamma$, and let $\Omega^*$ be a disk with $|\Omega|=|\Omega^*|$ and boundary $\Gamma^*$. Then
\begin{equation}\label{BE}
\lambda_1(\Gamma,\hat A)\leq\lambda_1(\Gamma^*,\hat A).
\end{equation}
If $\frac{\beta|\Omega|}{2\pi}\notin\mathbb N$, then equality holds if and only if $\Omega=\Omega^*$.
\end{thm}

\subsection{A few properties of magnetic eigenvalues}\label{properties}

We collect in this subsection a few properties of the eigenvalues of \eqref{AB_N_0} which will be useful in the sequel.

\medskip

First, we recall that $\lambda_j(\Omega,\beta)$ are invariant under isometries, namely, if $M$ is an isometry of $\mathbb R^2$, then
$$
\lambda_j(\Omega,\beta)=\lambda_j(M(\Omega),\beta).
$$
For the proof we refer to \cite[Appendix A]{laugesen_liang}.

\medskip
For any $\beta\in\mathbb R$ it is not difficult to show that
$$
\lambda_j(\Omega,\beta)=\lambda_j(\Omega,-\beta).
$$
The proof can be performed by observing that $u$ is an eigenfunction corresponding to $\lambda_j(\Omega,\beta)$ if and only if $\bar u$ is an eigenfunction corresponding to $\lambda_j(\Omega,-\beta)$. Therefore it is not restrictive to consider only positive values of $\beta$.
\medskip

The asymptotics of $\lambda_1(\Omega,\beta)$ for large magnetic field have been investigated in depth (see e.g., \cite[\S 8]{FH_book}). It turns out that if $\Omega$ is smooth, then
\begin{equation}\label{asympt_B}
\lim_{\beta\to+\infty}\frac{\lambda_1(\Omega,\beta)}{\beta}=\Theta_0
\end{equation}
where $\Theta_0>0$ is a universal constant (de Gennes constant) defined as
\begin{equation}\label{T00}
\Theta_0=\min_{\xi\in\mathbb R}\mu_1(\xi)=\mu_1(\xi_0),
\end{equation}
with $\mu_1(\xi)$ being the first eigenvalue of the following one-dimensional problem:
\begin{equation}\label{ODE}
\begin{cases}
-f''(t)+(\xi+t)^2f(t)=\mu(\xi)f(t)\,, & t\in(0,+\infty)\\
f'(0)=0.
\end{cases}
\end{equation}
For any $\xi\in\mathbb R$ it is standard to show that problem \eqref{ODE} admits a discrete spectrum made of a sequence of simple, non-negative eigenvalues diverging to $+\infty$.
It is known (see \cite{FH_book}) that
$$
\Theta_0=\xi_0^2\approx 0.590106,
$$
and that $\xi_0<0$. We may refer e.g., to \cite{bon} for the numerical approximation of $\Theta_0$ and for an estimate of the remainder.

\medskip
The limit \eqref{asympt_B} has a surprising consequence for the first eigenvalue of a family of homothetic domains. It follows from \eqref{minmax} (see also \cite{FH_book}) that for all $\alpha>0$
\begin{equation}\label{homothety}
\lambda_j(\Omega,\beta)=\alpha^2\lambda_j\left(\alpha\Omega,\frac{\beta}{\alpha^2}\right).
\end{equation}
From \eqref{homothety} we see that
$$
\lambda_1(\alpha\Omega,\beta)=\frac{1}{\alpha^2}\lambda_1(\Omega,\alpha^2\beta)=\Theta_0\beta+o(1)\,,\ \ \ {\rm as\ }\alpha\to +\infty.
$$
In particular, the asymptotic limit is strictly positive and does not depend on the measure of the domain.

\medskip

On the other hand, as $\alpha\to 0^+$, from \cite[\S 1.5]{FH_book} we have that there exists a potential $1$-form $A'$ and a constant $C_{\Omega}>0$, both depending on $\Omega$, such that
\begin{equation}\label{A0} 
\frac{\alpha^4\beta^2}{\vert \Omega \vert} \int_{\Omega} \vert A'\vert^2 -C_{\Omega}\alpha^8\beta^4 \le \lambda_1(\Omega,\alpha^2\beta) \le \frac{\alpha^4\beta^2}{\vert \Omega \vert} \int_{\Omega} \vert A'\vert^2.
\end{equation}
If $\Omega$ is simply connected, $A'=A_{can}$, where $A_{can}$ is a distinguished potential $1$-form which differs from $A$ by an exact $1$-form (see Section \ref{sub:variable} for the precise definition of $A_{can}$).

This implies that 

$$
\lambda_1(\alpha\Omega,\beta)=\frac{1}{\alpha^2} \lambda_1(\Omega,\alpha^2\beta)=O(\alpha^2)\,,\ \ \ {\rm as\ }\alpha\to 0^+.
$$

The peculiar behavior of magnetic Neumann eigenvalues is clearly highlighted in Figure \ref{F1}, were we have represented the analytic branches of the eigenvalues of the magnetic Laplacian with $\beta=1$ on the disk $B_R:=B(0,R)$ as functions of $R$. The first eigenvalue is singled out just by taking the minimum among all analytic branches. It vanishes as $R\to 0^+$ with quadratic speed, and shows an oscillating behavior as $R$ grows. It remains bounded and converges to $\Theta_0$ as $R\to+\infty$. For the disk $B_R$, from \eqref{A0} (see also \cite{FH2}) we can make the asymptotic behavior at $R=0$ more precise:
\begin{equation}\label{asympt_B_R_0}
\lambda_1(B_R,\beta)=\frac{\beta^2R^2}{8}+o(R^2)\,,\ \ \ {\rm as\ R\to 0^+}.
\end{equation}

\section{Upper bounds for $\lambda_1$}\label{sec:up}

In this section we establish upper bounds for $\lambda_1(\Omega,\beta)$. 

\medskip

In order to state our first result, we recall the definition of circumradius $R_{\Omega}$ of a domain $\Omega$:
$$
R_{\Omega}:=\inf\left\{R>0:{\rm there\ exists\ } x_R\in\mathbb R^2{\rm\ such\ that\ }\Omega\subset B(x_R,R)\right\}.
$$
We  have the following theorem, which implies Theorem \ref{TA}.

\begin{thm}\label{up_B_diam}
For any bounded domain $\Omega$ with circumradius $R_{\Omega}$ we have 
\begin{equation}\label{up_B_0}
\lambda_1(\Omega,\beta)\leq
\begin{cases}
\beta-\frac{1}{2R_{\Omega}^2}\,, & {\rm if\ }R_{\Omega}>\frac{1}{\sqrt{\beta}},\\
\frac{R_{\Omega}^2\beta^2}{2}\,, & {\rm if\ }R_{\Omega}\leq\frac{1}{\sqrt{\beta}}
\end{cases}
\end{equation}
In particular, if $R_{\Omega}\leq 1/\sqrt{\beta}$, then $\lambda_1(\Omega,\beta)\leq\beta/2$. It follows that, for all $\beta>0$
$$
\lambda_1(\Omega,\beta)<\beta.
$$
\end{thm}
\begin{proof}
Through the proof we shall denote $\lambda_1(\Omega,\beta)$ simply by $\lambda_1$.
Let $(r,t)$ denote the standard polar coordinates in $\mathbb R^2$, where $r=|x|$ and $t$ is the angular variable. We define the family of functions $\{u_n(r,t)\}_{n\in\mathbb N}$, expressed in polar coordinates, by setting $u_n(r,t):=r^ne^{in t}e^{-\frac{\beta r^2}{4}}$.  Recalling that in polar coordinates
\begin{equation}\label{laplacian_polar}
\Delta_Au=-\partial^2_{rr}u-\frac{\partial_r u}{r}-\frac{\partial^2_{tt}u}{r^2}+\frac{\beta^2r^2}{4}u+i\beta\partial_t u,
\end{equation}
it is standard to prove that $\Delta_Au_n=\beta u_n$. A standard computation (see also \cite{bauman}) shows that
$$
|\nabla^A u_n|^2=\beta|u_n|^2-\frac{1}{2}\Delta|u_n|^2.
$$
Hence, from the min-max principle \eqref{minmax} we find that for all $n\in\mathbb N$
 
\begin{equation}\label{av_0}
\lambda_1\int_{\Omega}|u_n|^2\leq\int_{\Omega}|\nabla^Au_n|^2=\beta\int_{\Omega}|u_n|^2-\frac{1}{2}\int_{\Omega}\Delta |u_n|^2.
\end{equation}
Now, if the last term of \eqref{av_0} has a negative sign, this would immediately imply that $\beta$ is a strict upper bound, but this is not in general the case.

\medskip

We start by proving the second inequality of \eqref{up_B_0}. Assume that, up to translations, $\Omega\subset B(0,R_{\Omega})$, and consider \eqref{av_0} with $n=0$. We have
$$
\lambda_1\int_{\Omega}e^{-\frac{\beta r^2}{2}}\leq\beta\int_{\Omega}e^{-\frac{\beta r^2}{2}}+\frac{1}{2}\int_{\Omega}\beta(r^2\beta-2)e^{-\frac{\beta r^2}{2}}=\frac{\beta^2}{2}\int_{\Omega}r^2e^{-\frac{\beta r^2}{2}}\leq\frac{R_{\Omega}^2\beta^2}{2}\int_{\Omega}e^{-\frac{\beta r^2}{2}}.
$$
This proves the second inequality in \eqref{up_B_0}. Note that this inequality is valid for any $\beta$, however it implies a strict upper bound by $\beta$ only for $R_{\Omega}<\sqrt{\frac{2}{\beta}}$.

\medskip 
We want to improve the upper bound for large $R_{\Omega}$ and to conclude the proof of \eqref{up_B_0}. The main idea behind the proof of the first inequality of \eqref{up_B_0} is to {\it average} inequality \eqref{av_0} with respect to $n$. Namely, we multiply both sides of \eqref{av_0} by some $a_n>0$, and sum the resulting inequalities over $n$, where $n$ ranges in some subset of $\mathbb N$. Choosing the weights $a_n$ in a suitable way, we will be able to make the sum of the terms involving $\Delta|u_n|^2$ at the right-hand side of \eqref{av_0} negative, in a controlled way.
Let then $a_n>0$, $n=0,...,N$. From \eqref{av_0} we get
$$
\lambda_1\int_{\Omega}\sum_{n=0}^Na_n|u_n|^2\leq\beta\int_{\Omega}\sum_{n=0}^Na_n|u_n|^2-\frac{1}{2}\int_{\Omega}\Delta\left(\sum_{n=0}^Na_n |u_n|^2\right),
$$
which implies
\begin{equation}\label{add}
\lambda_1\leq\beta-\frac{1}{2}\frac{\int_{\Omega}\Delta\left(\sum_{n=0}^Na_n |u_n|^2\right)}{\int_{\Omega}\sum_{n=0}^Na_n|u_n|^2}.
\end{equation}
Now, we note that

$$
|u_n|^2=e^{-\frac{\beta r^2}{2}}r^{2n},
$$
hence
$$
\sum_{n=0}^Na_n |u_n|^2=e^{-\frac{\beta r^2}{2}}\sum_{n=0}^Na_nr^{2n}.
$$
The scope is now to choose suitable $a_n>0$ and take the limit as $N\to\infty$. This is done by noting that
$$
e^{\frac{\beta r^2}{2}}=\sum_{n=0}^{\infty}\frac{\beta^n}{2^nn!}r^{2n},
$$
and the convergence is uniform on any compact subset of $\mathbb R^2$. Then we choose
$$
a_n=\frac{\beta^n}{c^nn!}
$$
with $c>0$.
Hence, on any compact set, we have
\begin{equation}\label{add0}
\lim_{N\to+\infty}\sum_{n=0}^Na_n|u_n|^2=e^{\frac{\beta r^2}{2c}(2-c)}
\end{equation}
and
$$
\lim_{N\to+\infty}-\Delta\left(\sum_{n=0}^Na_n|u_n|^2\right)=e^{\frac{\beta r^2}{2c}(2-c)}\beta\frac{(2c+(2-c)r^2\beta)(2-c)}{c^2}
$$
Now, we assume that $\Omega\subset B(0,R_{\Omega})$. Then we have, for $r\leq R_{\Omega}$
$$
\frac{(2c+(2-c)r^2\beta)(2-c)}{c^2}=\frac{2(2-c)}{c}+\frac{(2-c)^2r^2\beta}{c^2}\leq \frac{2(2-c)}{c}+\frac{(2-c)^2R_{\Omega}^2\beta}{c^2}.
$$
Consider now the function
$$
g(c)=\frac{2(2-c)}{c}+\frac{(2-c)^2R_{\Omega}^2\beta}{c^2}.
$$
We have that
$$
g'(c)=-\frac{4}{c^3}(c+(2-c)R_{\Omega}^2)
$$
and
$$
\lim_{c\to 0^+}g(c)=+\infty\,,\ \ \ \lim_{c\to +\infty}g(c)=\beta R_{\Omega}^2-2.
$$
We see that
$$
g'(c)=0{\rm\ \ \ \iff\ \ \ }c=\frac{2R_{\Omega}^2\beta}{R_{\Omega}^2\beta-1}.
$$
Now, if $R_{\Omega}>\frac{1}{\sqrt{\beta}}$, we choose $c=\frac{2R_{\Omega}^2\beta}{R_{\Omega}^2\beta-1}$
and with this choice $g\left(\frac{2R_{\Omega}^2\beta}{R_{\Omega}^2\beta-1}\right)=-\frac{1}{\beta R_{\Omega}^2}$. We conclude that, when $c=\frac{2R_{\Omega}^2\beta}{R_{\Omega}^2\beta-1}$,
\begin{multline}\label{add1}
\lim_{N\to+\infty}-\Delta\left(\sum_{n=0}^Na_n|u_n|^2\right)=e^{\frac{\beta r^2}{2c}(2-c)}\beta\frac{(2c+(2-c)r^2\beta)(2-c)}{c^2}
\leq -\frac{1}{R_{\Omega}^2}e^{\frac{\beta r^2}{2c}(2-c)}.
\end{multline}
Using \eqref{add0} and \eqref{add1} in \eqref{add} we deduce the first inequality of \eqref{up_B_0}.

\end{proof}

\begin{rem}
The upper bound for $R_{\Omega}\leq\frac{1}{\sqrt{\beta}}$ shows a correct behavior with respect to $\beta$, which is quadratic, in view of \eqref{homothety}, \eqref{A0} and \eqref{asympt_B_R_0}.
\end{rem}

In view of the asymptotic behavior \eqref{asympt_B} the natural question is whether $\Theta_0\beta$ is an upper bound for $\lambda_1(\Omega,\beta)$, for any domain $\Omega$. We prove this result for certain classes of domains. 

\begin{thm}\label{T0_1}
Let $\Omega$ be a bounded domain of $\mathbb R^2$ satisfying (up to isometries) one of the following two conditions:
\begin{enumerate}[1)]
\item $\Omega$ is a sub-graph, namely
$$
\Omega=\{(x_1,x_2)\in\mathbb R^2:a<x_1<b\,, 0< x_2<  g(x_1)\}
$$
for some smooth $g:(a,b)\to[0,+\infty)$, $-\infty<a<b<+\infty$.
\item $\Omega$ is contained in some strip $(a,b)\times(0,+\infty)$ and contains $(a,b)\times(0,-2\xi_0/\sqrt{\beta})$, where $\xi_0<0$ is the constant defined in \eqref{T00}.
\end{enumerate}
Then
\begin{equation}\label{ineq_T0_1}
\lambda_1(\Omega,\beta)<\Theta_0\beta.
\end{equation}
\end{thm}
\begin{proof}We first remark that it is sufficient to prove the result for $\beta=1$. In fact, from \eqref{homothety} we have that $\lambda_1(\Omega,\beta)=\beta\lambda_1(\Omega',1)$, where $\Omega'=\sqrt{\beta}\Omega$. Now, $\Omega$ is a sub-graph of the form 1) if and only if $\Omega'$ is; $\Omega$ satisfies condition 2) if and only if $\Omega'$ does  with $\beta=1$.

\medskip

Let $f$ be a first eigenfunction of \eqref{ODE} with $\xi=\xi_0<0$ and hence first eigenvalue $\Theta_0$, defined in \eqref{T00}. Recall that $\Theta_0=\xi_0^2$. We can choose $f>0$ on $[0,+\infty)$. We recall that $f$ satisfies
\begin{equation}\label{ODE2}
-f''(t)+(\xi_0+t)^2f(t)=\Theta_0f(t)\,, \ t\in(0,+\infty)\,,\ \ \ f'(0)=0.
\end{equation}
We prove that $f'<0$ on $(0,+\infty)$ and that $\lim_{t\to+\infty}f(t)=\lim_{t\to+\infty}f'(t)=0$.
Equation \eqref{ODE2} implies that $f''$ has only one zero in $(0,+\infty)$, namely $-2\xi_0$ (moreover, $f''(0)=0$). This implies that $f'$ is monotone on $(-2\xi_0,+\infty)$, and since $f\in H^1((0,+\infty))$, necessarily $f'$ is increasing to $0$ on $(-2\xi_0,+\infty)$ and  $\lim_{t\to\infty}f'(t)=0$. On the other hand, $f'$ is decreasing on $(0,-2\xi_0)$. Moreover, $f'(0)=0$ and $f(0)>0$. Suppose by contradiction that $f$ is not decreasing on $(0,+\infty)$. This would imply the existence of $t_0\in(0,+\infty)$ such that $f'(t_0)=0$. Since $f'(0)=0$ and $\lim_{t\to+\infty}f'(t)=0$, this would imply the existence of two distinct points $t_1, t_2\in(0,+\infty)$ such that $f''(t_1)=f''(t_2)=0$, but this is impossible since $f''(t)=0$ on $(0,+\infty)$ if and only if $t=-2\xi_0$. This proves that $f'<0$ on $(0,+\infty)$. In particular then, $\lim_{t\to+\infty}f(t)=0$, since $f\in L^2((0,+\infty))$.
\medskip

We consider the magnetic Laplacian $\Delta_{A'}$, where $A'=-(\xi_0+x_2)dx_1$. By gauge invariance, the Neumann spectrum of $\Delta_{A'}$ coincides with the Neumann spectrum of $\Delta_A$ on $\Omega$ (the two forms differ by an exact 1-form). From now on we shall denote $\lambda_1(1,\Omega)$ simply by $\lambda_1$
\medskip

We are ready to prove 1) for $\beta=1$.  Using $f(x_2)$ as test function in \eqref{minmax} we get
\begin{multline*}
\lambda_1\int_{\Omega}f^2\leq\int_{\Omega}|\nabla f|^2+(\xi_0+x_2)^2f^2=\int_a^b\int_0^{g(x_1)}(f'(x_2)^2+(\xi_0+x_2)^2f^2(x_2))dx_2dx_1\\
=\Theta_0\int_{\Omega}f^2+\int_a^b f(g(x_1))f'(g(x_1))dx_1<\Theta_0\int_{\Omega}f^2.
\end{multline*}
We have used, in the integration by parts, the fact that $f'(0)=0$; the last inequality follows since $f$ is positive and strictly decreasing.

\medskip

In order to prove 2) (and re-prove 1)), we use $f(x_2)$ as test function in \eqref{minmax}, and use the identity $f'^2=-ff''+\frac{1}{2}(f^2)''$. We obtain

\begin{multline}\label{Alessandro1}
\lambda_1\int_{\Omega}f^2\leq\int_{\Omega}|\nabla f|^2+(\xi_0+x_2)^2f^2=\int_{\Omega}(-ff''+(\xi_0+x_2)f^2)+\frac{1}{2}\int_{\Omega}(f^2(x_2))''\\=\Theta_0\int_{\Omega}f^2+\frac{1}{2}\int_{\Omega}(f^2)''
\end{multline}
We are left with the study $\frac{1}{2}\int_{\Omega}(f^2)''$. In fact the upper bound $\lambda_1<\Theta_0$ holds provided $\frac{1}{2}\int_{\Omega}(f^2)''< 0$.

We see that
$$
\frac{1}{2}\int_0^{+\infty}(f^2(t))''=\frac{1}{2}\lim_{s\to+\infty}\int_0^s(f^2(t))''=\lim_{s\to+\infty}\frac{1}{2}(f^2(s))'=\lim_{t\to+\infty}f(s)f'(s)=0
$$
since $f'(0)=0$. Moreover, using the identity $f'^2=-ff''+\frac{1}{2}(f^2)''$ and the differential equation satisfied by $f$, we find that
$$
\frac{1}{2}(f^2(t))''=(f'(t))^2+tf^2(t)(t+2\xi_0)
$$
and this quantity is non-negative for $t\geq-2\xi_0$.

We easily deduce two facts:
\begin{enumerate}[a)]
\item $\frac{1}{2}\int_0^L(f^2(t))''<0$ for all $L>0$.
\item $\frac{1}{2}\int_I(f^2(t))''<0$ for all $I\subset\mathbb R^+$ such that $(0,-2\xi_0)\subset I$.
\end{enumerate}

Either a) or b) imply the inequality \eqref{ineq_T0_1}. We show now that 1) and 2) in the statement of Theorem \ref{T0_1} imply a) and b), respectively. In fact, we can re-write the last term of \eqref{Alessandro1} as
$$
\frac{1}{2}\int_{\Omega}(f^2)''=\int_{P(\Omega)}\left(\frac{1}{2}\int_{E(x_1)}(f^2(x_2))''dx_2\right)dx_1
$$
where $P(\Omega)=\{x_1\in\mathbb R:(x_1,x_2)\in\Omega\}\subset\mathbb R$ and, for any $x_1\in P(\Omega)$, $E(x_1)=\{x_2\in(0,+\infty):(x_1,x_2)\in\Omega\}$. 
Recall that we have assumed that $\Omega$ is in the half-plane $x_2>0$. Assume we are in the hypothesis 1) or 2):
\begin{enumerate}[1)]
\item $\Omega$ is a sub-graph. Then we have $E(x_1)=(0,L(x_1))$ for all $x_1$, $L(x_1)>0$, which is a).
\item $\Omega$ is contained in some strip $(a,b)\times(0,+\infty)$ and $\Omega$ contains $(a,b)\times(0,-2\xi_0)$, which is b).
\end{enumerate}
The proof is now concluded.


\end{proof}

\begin{rem}
Note that the case 2) of Theorem \ref{T0_1} implies an upper bound for $\lambda_1$ with $\Theta_0\beta$ for a class of domains containing also non-simply connected domains. We just require that a suitable rectangle is contained in the domain. Note that as $\beta\to+\infty$ the size of the rectangle becomes small, hence more domains are allowed for the upper bound. As $\beta\to 0^+$, less domains are allowed. However, as $\beta\to 0^+$ we have in general better upper bounds than $\Theta_0\beta$, in fact upper bounds behave like $C\beta^2$ as $\beta\to 0^+$ (see \eqref{A0}, see also \cite{FH_book}). Note also that  suitable unions of domains of the form 1) and 2) still enjoy the upper bound $\Theta_0\beta$.
\end{rem}

We prove now a similar upper bound for {\it self-tiling} domains.

\begin{defn} A piecewise smooth domain $\Omega\subset\mathbb R^2$ is called \emph{self-tiling} if there exists an integer $m\ge2$ such that 
\[\Omega={\rm Int}\left(\bigcup_{i=1}^m\overline{\omega_i}\right),\]
where each $\omega_i$ is isometric to $\frac{1}{\sqrt{m}}\Omega$.
\end{defn}

\begin{rem}
	The definition is of course restrictive. Nevertheless, all triangles and all parallelograms are self-tiling.  Note that not all self-tiling domains are covered by the previous Theorem \ref{T0_1}, e.g., parallelograms are not.
\end{rem}

We have already recalled that for any smooth bounded  domain, the limit \eqref{asympt_B} holds. We now assume  that the boundary of $\Omega$ is a curvilinear polygon, that is to say 
\begin{equation*}
						\partial \Omega=\bigcup_{s=1}^N\gamma_s,\end{equation*}
the $\gamma_s$ being $C^\infty$-arcs which are disjoint, except at the endpoints, where $\gamma_{s-1}$ and $\gamma_s$ meet with an angle $\alpha_s\in]0,2\pi[$ (using the convention $\gamma_0=\gamma_N$).  Under this assumption, it follows from \cite[Corollary 1.3]{vir1} that  there exists a constant $\Lambda(\Omega)\le \Theta_0$, depending only on the angles $\alpha_s$, such that
\begin{equation*}
	\lim_{\beta\to +\infty}\frac{\lambda_1(\Omega,\beta)}{\beta}=\Lambda(\Omega).
\end{equation*}
Furthermore, as described in \cite[Remarks 2.6 and 4.3]{vir1}, $\Lambda(\Omega)<\Theta_0$ whenever $\min_{1\le s\le N}\alpha_s\le\frac\pi2$. This is in particular the case when $\Omega$ is a triangle or a parallelogram. 

\medskip

We are now ready to state the next theorem.

\begin{thm}\label{propUpper} Let $\Omega$ be a self-tiling curvilinear polygon. Then 
\begin{equation}\label{ineq_T_1}
\lambda_1(\Omega,\beta)\leq \Lambda(\Omega)\,\beta.
\end{equation}
\end{thm}
\begin{proof} 
The proof is an adaptation of the argument from Pólya (see \cite{polyatiling}). We fix one of the pieces in the decomposition of $\Omega$, say $\omega_1$. For all $\beta\in\mathbb R$, 
\begin{equation}\label{eq:var}
\lambda_1(\omega_1,\beta)\le \lambda_1(\Omega,\beta).
\end{equation}
Indeed, the Sobolev space $H^1(\Omega)$ can be seen as  a subspace of $H^1(\Omega')$, with
 \[\Omega':=\bigcup_{i=1}^m \omega_i. {\rm \ \ \ (the\ union\ is\ disjoint)}\] 
The Hilbert space $H^1(\Omega')$ can itself be seen as the direct sum
 \[\bigoplus_{i=1}^m H^1(\omega_i).\]
This last identification tells us that $\lambda_1(\Omega',\beta)=\lambda_1(\omega_1,\beta)$, and the inclusion $H^1(\Omega)\subset H^1(\Omega')$ implies, by the variational definition of eigenvalues, that 
\[\lambda_1(\Omega',\beta)\le \lambda_1(\Omega,\beta).\]
This yields \eqref{eq:var}. In particular, replacing $\beta$ by $m\beta$ we have
$$
\lambda_1(\omega_1,m\beta)\leq\lambda_1(\Omega,m\beta)
$$
for all $m>0$. On the other hand, the scaling property of the magnetic eigenvalues tells us that
\begin{equation}\label{eq:scal}
\lambda_1(\omega_1,m\beta)=\lambda_1\left(\frac{1}{\sqrt{m}}\Omega,m\beta\right)=m\,\lambda_1\left(\Omega,\beta\right).
\end{equation}
We get
\begin{equation}\label{eq:step}
	\lambda_1(\Omega,\beta)\le \frac1{m}\lambda_1(\Omega,m\beta).
\end{equation}
Iterating \eqref{eq:step}, we obtain that for all positive integers $k$,
\begin{equation}
	\lambda_1(\Omega,\beta)\le \frac1{m^k}\lambda_1\left(\Omega,m^k\beta\right).
\end{equation}
Taking $k\to+\infty$, we find
\[\frac{\lambda_1(\Omega,\beta)}\beta\le \liminf_{k\to+\infty}\frac{\lambda_1\left(\Omega,m^k\beta\right)}{m^k\,\beta}\le \Lambda(\Omega).\qedhere\]
This concludes the proof.
\end{proof}

\section{Upper bounds of $\lambda_1$ for variable magnetic fields}\label{sub:variable}
In this section we consider the first eigenvalue of problem \eqref{AB_N_0} when $\Omega$ is a planar domain or more generally an orientable compact surface with boundary, and $A$ is a generic smooth magnetic potential (a smooth $1$-form) giving rise to a (variable) magnetic field $dA=\beta dv$, where $dv$ is the Riemannian volume form of $\Omega$. Through this section we shall denote the first eigenvalue of \eqref{AB_N_0} by $\lambda_1(\Omega,A)$. As we have already discussed in Subsection \ref{properties}, if $\Omega$ is simply connected, $\lambda_1(\Omega,A)$ depends only on $\beta$. In the case that $\Omega$ is not simply connected, given $A,A'$ with $dA=dA'=\beta dv$, in general $\lambda_1(\Omega,A)\ne\lambda_1(\Omega,A')$. In this case we shall choose a distinguished {\it primitive} of $\beta dv$, which we will call $A_{can}$. In order to define $A_{can}$, we need a few preliminaries.

Let $\Omega$ be a compact orientable surface with boundary, and let $\beta:\Omega\to\mathbb R$ be a given smooth function. Consider problem
\begin{equation}\label{torsion}
\begin{cases}
\Delta	\phi=\beta\,, & {\rm in\ }\Omega,\\
\phi=0\,, & {\rm on\ }\partial\Omega.
\end{cases}
\end{equation}
Problem \eqref{torsion} admits a unique solution which reduces to the torsion function when $\beta=1$. We call $A_{can}$  the $1$-form defined by $A_{can}=-\star d\phi$, where $\star$ is the Hodge-star operator acting on differential forms, for a chosen orientation of $\Omega$. Recall that the Hodge-star operator is defined by the following relation: for any pair of $1$-forms $\psi_1,\psi_2$ we have
$$
\psi_1\wedge\star\psi_2:=\langle\psi_1,\psi_2\rangle dv,
$$
and $\star 1=dv$, $\star dv=1$. For example, in $\mathbb R^2$ $\langle\cdot,\cdot\rangle$ stands for the standard scalar product and, with Cartesian coordinates $(x_1,x_2)$ and positive orthonormal basis $\left(\frac{\partial}{\partial x_1},\frac{\partial}{\partial x_2}\right)$, one has $\star 1=dx_1\wedge dx_2$, $\star dx_1=dx_2$, $\star dx_2=-dx_1$, $\star(dx_1\wedge dx_2)=1$. We denote by $\delta$ the co-differential (on $1$-forms we have $\delta=-{\rm div}$).
We prove the following lemma.
\begin{lem} The form $A_{can}$ is a primitive of $\beta dv$, i.e., $dA_{can}=\beta dv$. 
\end{lem}
\begin{proof}  We recall that $\delta=-\star d\star$, which implies that, for $1$-forms, $\star\delta=-d\star$ and $\delta\star=-\star d$.  Consider the $1$-form $A_{can}=-\star d\phi$, where $\phi$ solves \eqref{torsion}. Then
$$
dA_{can}=-d\star d\phi=\star\delta d\phi=\star \Delta\phi=\star\beta=\beta dv.
$$ 
\end{proof} 
We are ready to state the main result of this section.
\begin{thm}\label{variable}
Let $\Omega$ be a simply connected, orientable, compact Riemannian surface and let $\beta:\Omega\to\mathbb R$ be a smooth function.  Let $A$ be any potential $1$-form such that $dA=\beta dv$ and let $\lambda_1(\Omega,A)$ denote the first eigenvalue of \eqref{AB_N_0} with magnetic potential $A$. Let $\phi:\Omega\to\mathbb R$ be the unique solution to \eqref{torsion}. Then
\begin{equation}\label{bound_general}
\lambda_1(\Omega,A)\leq\frac{\int_{\Omega}\beta(e^{2\phi}-1)dv}{\int_{\Omega}e^{2\phi}dv}.
\end{equation}
If $\Omega$ is not simply connected, the inequality holds for $\lambda_1(\Omega,A_{can})$.
\end{thm}
\begin{proof} Clearly, it is enough to show the assertion for $A=A_{can}=-\star d\phi$. 
Note that then, since the Hodge star operator is an isometry, we have $\abs{A}^2=\abs{d\phi}^2$. If $u$ is a real valued smooth function then we have from \eqref{minmax}
\begin{equation}\label{minmaxV}
\lambda_1(\Omega,A)\int_{\Omega}u^2\leq \int_{\Omega}\abs{\nabla u}^2+\abs{A}^2u^2
=\int_{\Omega}\abs{\nabla u}^2+\abs{\nabla\phi}^2u^2
\end{equation}
We take $u=e^{\phi}$ so that $\Delta u=-u\abs{\nabla\phi}^2+u\Delta\phi=-u\abs{\nabla\phi}^2+u\beta$. Integrating by parts, taking into account that on $\partial\Omega$ one has $u=1$ and $\frac{\partial u}{\partial N}=\frac{\partial \phi}{\partial N}$, and using the Green formula on the last boundary integral, we get:
$$
\int_{\Omega}\abs{\nabla u}^2=\int_{\Omega}u\Delta u+\int_{\partial\Omega}u\frac{\partial u}{\partial N}=
-\int_{\Omega}\abs{\nabla\phi}^2u^2+\int_{\Omega}\beta(u^2-1).
$$
Inserting this identity in the right-hand side of \eqref{minmaxV} we obtain the assertion. 
\end{proof}
 Let us set $\beta^*=\max_{\overline{\Omega}}|\beta|$, $\phi^*=\max_{\overline{\Omega}}|\phi|$. From Theorem \ref{variable} we deduce the following

\begin{cor}\label{variable_1}
Let $\Omega$ be a simply connected, orientable, compact Riemannian surface, let $A$ be any smooth $1$-form such that $dA=\beta dv$, with $\beta:\Omega\to\mathbb R$ smooth.  Then
\begin{enumerate}[1)]
\item
\begin{equation}\label{bound_general_0}
\lambda_1(\Omega,A)<\beta^*.
\end{equation}
\item If $\beta\geq 0$, then
\begin{equation}\label{bound_general_1}
\lambda_1(\Omega,A)\leq\beta^*(1-e^{-2\phi^*}).
\end{equation}
\item
If $\beta\geq 0$ and $\Omega$ is a domain in $\mathbb R^2$, then
\begin{equation}\label{bound_general_2}
\lambda_1(\Omega,A)\leq\beta^*(1-e^{-\frac{\beta^*|\Omega|}{2\pi}}).
\end{equation}
\item If $\Omega$ is a domain in $\mathbb R^2$ and $A$ is given by \eqref{magnetic_potential} (hence $\beta(x)\equiv\beta>0$), then
\begin{equation}\label{bound_general_3}
\lambda_1(\Omega,A)=\lambda_1(\Omega,\beta)\leq\beta(1-e^{-\frac{\beta|\Omega|}{2\pi}}).
\end{equation}
\item If $\Omega$ is not simply connected, all the inequalities above hold with $A=A_{can}$.
\end{enumerate}
\end{cor}
\begin{proof}  We start by proving 1). We note that replacing $A$ by $-A$, the eigenfunctions of problem \eqref{AB_N_0} are changed to their conjugates, but the spectrum remains the same.  From Theorem \ref{variable} we have
$$
\lambda_1(\Omega,A)\int_{\Omega}e^{2\phi}\leq \int_{\Omega}\beta(e^{2\phi}-1)
$$
and, at the same time, changing $\beta$ to $-\beta$ and $\phi$ to $-\phi$:
$$
\lambda_1(\Omega,A)\int_{\Omega}e^{-2\phi}\leq \int_{\Omega}-\beta(e^{-2\phi}-1) .
$$
Summing up the two inequalities we obtain
$$
\lambda_1(\Omega,A)\int_{\Omega}\cosh(2\phi)\leq \int_{\Omega}\beta\sinh(2\phi)\leq \beta^*\int_{\Omega}\abs{\sinh(2\phi)}
= \beta^*\int_{\Omega}\sinh(\abs{2\phi})<\beta^*\int_{\Omega}\cosh(2\phi).
$$
This proves \eqref{bound_general_0}.

\medskip

We pass to 2). Inequality \eqref{bound_general_1} follows immediately from \eqref{bound_general}: from the maximum principle, since $\beta\geq 0$, we have that $\phi\geq 0$, and we conclude by rough estimates.

\medskip

We prove 3). Inequality \eqref{bound_general_2}, is a consequence of the well-known isoperimetric inequality
$$
\max_{\overline{\Omega}}\phi_{\Omega}\leq \max_{\overline{B_R}}\phi_{B_R},
$$
where $\phi_{\Omega},\phi_{B_R}$ are the solutions of \eqref{torsion} with $\beta=1$ on $\Omega$ and $B_R$, respectively, see \cite{talenti}. Here $\Omega\subset\mathbb R^2$, and $B_R\subset\mathbb R^2$ is a disk of radius $R$ centered at $0$ with $|B_R|=|\Omega|$. In fact, $\phi_{R}(x)=\frac{1}{4}(R^2-|x|^2)$, and therefore $\phi_{B_R}^*:=\max_{\overline{B_R}}\phi_{B_R}=\frac{|\Omega|}{4\pi}$. Now, observe that $\Delta(\beta^*\phi_{\Omega})=\beta^*\geq\beta$ on $\Omega$ and $\beta^*\phi_{\Omega}=0$ on $\partial\Omega$. Hence, by the maximum principle we have $\phi\leq\beta^*\phi_{\Omega}$ on $\Omega$ and hence
$$
\phi^*=\max_{\overline{\Omega}}\phi\leq\beta^*\max_{\overline\Omega}\phi_{\Omega}\leq\frac{\beta^*|\Omega|}{4\pi}.
$$
\medskip

Next we consider 4):  inequality \eqref{bound_general_3} follows immediately from \eqref{bound_general_2}.

\medskip

The assertion 5) is straightforward. This concludes the proof.
\end{proof}

\section{Lower bounds for $\lambda_1$}\label{sec:lo}

The study of lower bounds for $\lambda_1(\Omega,\beta)$ is rather challenging. It is easy to produce small eigenvalues by  perturbing a domain $\Omega$ with a local perturbation near the boundary, namely, attaching to $\Omega$ a small Cheeger dumbbell, as for the usual Neumann problem for the Laplacian. On the other hand, one can produce examples of convex domains of large diameter and any measure with first eigenvalue either arbitrarily small or bounded away from zero. Also, there exist thin domains with first eigenvalue arbitrarily small or uniformly bounded away from zero as the thickness goes to zero. We have collected a series of  examples in Appendix \ref{sub:examples}.

\medskip The starting point of our analysis is \cite[Theorem 5.1]{kovarik}, which provides lower bounds for $\lambda_1(\Omega,\beta)$ in terms of $\beta,|\Omega|$, $\lambda_2^N(\Omega)$ (the second Neumann eigenvalue of the Laplacian on $\Omega$), and the inradius $\rho_{\Omega}$ of $\Omega$, defined by
\begin{equation}\label{inradius}
\rho_{\Omega}:=\sup_{x\in\Omega}\inf_{y\in\partial\Omega}|x-y|
\end{equation}
We recall it here for the reader's convenience.
\begin{thm}[{\cite[Theorem 5.1]{kovarik}}]\label{kov1}
Let $\Omega$ be bounded, simply connected domain in $\mathbb R^2$. Then
\begin{equation}\label{ko1}
\lambda_1(\Omega,\beta)\geq\frac{\pi}{4|\Omega|}\cdot \frac{\beta^2\rho_{\Omega}^4\lambda_2^N(\Omega)}{\beta^2\rho_{\Omega}^2+6\lambda_2^N(\Omega)}\,,\ \ \ {\rm if\ }\beta\leq\rho_{\Omega}^{-2}
\end{equation}
and
\begin{equation}\label{ko2}
\lambda_1(\Omega,\beta)\geq\frac{\pi}{4|\Omega|}\cdot \frac{\beta\rho_{\Omega}^2\lambda_2^N(\Omega)}{\beta+24\lambda_2^N(\Omega)}\,,\ \ \ {\rm if\ }\beta\geq\rho_{\Omega}^{-2},
\end{equation}
where $\lambda_2^N(\Omega)$ is the first positive eigenvalue of the Neumann Laplacian on $\Omega$ and $\rho_{\Omega}$ is the inradius of $\Omega$.
\end {thm}
Note that for some domains the lower bounds given by Theorem \ref{kov1} are not optimal. The following example clarifies this.

 \begin{ex}\label{ex_kov}
 Let $\Omega=]-k,k[\times]-1/2,1/2[$, with $k\in\mathbb N$. Then $\rho_{\Omega}=1/4$. As $k\to+\infty$, $\lambda_2^N(\Omega)\to 0$, and therefore the lower bound given by \eqref{ko1}-\eqref{ko2} goes to $0$ as well. On the other hand, we have 
$$
\Omega={\rm Int}\bigcup_{l=-k}^k\overline {\Omega_l}
$$
where
$$
\Omega_l=]l,(l+1)[\times]-1/2,1/2[.
$$
Clearly, bounds \eqref{ko1}-\eqref{ko2} hold for  $\lambda_1(\Omega_l,\beta)$ with $\rho_{\Omega_l}=1/4$, $\lambda_2^N(\Omega_l)=\pi^2$ and $|\Omega_l|=1$. Therefore the same lower bound holds for $\Omega$ (see Theorem \ref{ko_imp}), and this lower bound does not depend on $k$, therefore is uniformly bounded away from $0$ as $k\to+\infty$.
\end{ex}
The direct application of Theorem \ref{kov1} to the previous example does not yield a good lower bound since the behavior of $\lambda_2^N(\Omega)$ and $\lambda_1(\Omega,\beta)$ drastically diverge as $k\to+\infty$: $\lambda_2^N(\Omega)$  vanishes (the area goes to $+\infty$) while $\lambda_1(\Omega,\beta)$ stays uniformly bounded away from zero (large area does not imply small eigenvalue). However, the use of a suitable covering of $\Omega$ and the application of Theorem \ref{kov1} on each piece of the covering, allow to improve the lower bound. This is the main idea behind the main result of this section. Before stating it, we need some preliminary results.

\begin{thm}\label{ko_imp}
Let $\Omega$ be a smooth bounded domain in $\mathbb R^2$ such that
$$
\Omega={\rm Int}\bigcup_{i=1}^N\overline{\Omega_i},
$$
where $\Omega_i\subset\Omega$ are subdomains such that a point $p \in \Omega$ is contained in at most $K$ subdomains. Then
\begin{equation}\label{lo_imp}
\lambda_1(\Omega,\beta)\geq\frac{1}{K}\min_{i=1,...,N}\lambda_1(\Omega_i,\beta).
\end{equation}
\end{thm}
\begin{proof}
Let $u$ be an eigenfunction associated with $\lambda_1(\Omega,\beta)$. Then its restriction to $\Omega_i$ is a suitable test function for the min-max principle \eqref{minmax} for $\lambda_1(\Omega_i,\beta)$, which means:
$$
\lambda_1(\Omega_i,\beta)\int_{\Omega_i}|u|^2\leq\int_{\Omega_i}|\nabla^A u|^2.
$$
Summing over $i=1,...,N$ we get
\begin{equation}\label{sum0}
\sum_{i=1}^N\lambda_1(\Omega_i,\beta)\int_{\Omega_i}|u|^2\leq\sum_{i=1}^N\int_{\Omega_i}|\nabla^A u|^2.
\end{equation}
For the left-hand side of \eqref{sum0} we have
\begin{equation}\label{sum1}
\sum_{i=1}^N\lambda_1(\Omega_i,\beta)\int_{\Omega_i}|u|^2\geq \min_{j=1,...,N}\lambda_1(\Omega_j,\beta)\sum_{i=1}^N\int_{\Omega_i}|u|^2\geq \min_{j=1,...,N}\lambda_1(\Omega_j,\beta)\int_{\Omega}|u|^2,
\end{equation}
while for the right-hand side of \eqref{sum0} we have
\begin{equation}\label{sum2}
\sum_{i=1}^N \int_{\Omega_i} \vert \nabla^A u \vert^2 \le  K\int_{\Omega} \vert \nabla^A u \vert^2.
\end{equation}
Thanks to \eqref{sum1}, \eqref{sum2} and \eqref{sum0} we get
$$
\min_{j=1,...,N}\lambda_1(\Omega_j,\beta)\int_{\Omega}|u|^2\leq K\int_{\Omega} \vert \nabla^A u \vert^2.
$$
It follows that
$$
\frac{\int_{\Omega}\vert  \nabla^A   u\vert^2}{\int_{\Omega} \vert u\vert^2} \ge \frac{1}{K}\min_{j=1,...,N}\lambda_1(\Omega_j,\beta).
$$
 \end{proof}

We will look for coverings $\Omega={\rm Int}\bigcup_{i=1}^N\overline{\Omega_i}$, where each $\Omega_i$ is star-shaped with respect to some of its points. In order to combine Theorems \ref{kov1} and \ref{ko_imp} we need to have good estimates for $\lambda_2^N(\Omega_i)$. We can get it as a special case of a result of \cite{ChLi1997}:
\begin{thm}[{\cite[Theorem 1]{ChLi1997}}] \label{ChenLi} Let $\Omega \subset \mathbb R^2$ be a bounded domain. Assume that $\Omega$ is star-shaped with respect to a point $p \in \Omega$. Let $R$ be the radius of the largest ball centered at $p$ contained in $\Omega$ and $R_0$ be the radius of the smallest ball centered at $p$ containing $\Omega$. There exists a universal constant $C_1>0$ such that the first nonzero Neumann eigenvalue $\lambda_2^N(\Omega)$ has a lower bound given by
\begin{equation} \label{lowerneumann}
    \lambda_2^N(\Omega) \ge C_1\frac{R^2}{R_0^4}.
\end{equation}
    \end{thm}

\begin{rem}
The result of \cite{ChLi1997} is stated for compact manifold with smooth boundary. In particular it holds for smooth Euclidean domains. However, it is valid also for piecewise smooth, Lipschitz domains by approximation with smooth domains. 
\end{rem}

A combination of Theorems \ref{kov1} and \ref{ChenLi} allows to prove a lower bound for $\lambda_1(\Omega,\beta)$ for star-shaped domains in term of the inner and the outer radius.

\begin{prop}\label{lo_star}
Let $\Omega$ be domain which is star-shaped with respect to $p$, with $p$ as in Theorem \ref{ChenLi}, and such that $B(p,R)\subset\Omega\subset B(p,R_0)$ for some  $0<R<R_0$. Then
$$
\lambda_1(\Omega,\beta) \ge c\beta^2\frac{R^8}{R_0^6}\,,\ \ \ {\rm if\ }\beta\leq\rho_{\Omega}^{-2}
$$
and
$$
\lambda_1(\Omega,\beta) \ge c\frac{R^6\beta}{R_0^6(R^2\beta+1)}\,,\ \ \ {\rm if\ }\beta\geq\rho_{\Omega}^{-2},
$$
for some universal constant $c>0$.
\end{prop}
\begin{proof} 
By Theorem \ref{ChenLi} we have
$$
\lambda_2^N(\Omega) \ge C_1\frac{R^2}{R_0^4}.
$$
Moreover
$$
\lambda_2^N(\Omega)\leq\lambda_1^D(\Omega)\leq\frac{\gamma}{R^2}\,,\ \ \ \frac{\pi}{4 \vert\Omega\vert }\ge \frac{1}{4R_0^2}\,,\ \ \ \rho_{\Omega}\geq R,
$$
where
$\gamma:=\lambda_1^D(B(0,1))$ denotes the first Dirichlet eigenvalue of the disk of radius $1$. Recall that $\lambda_2^N(\Omega)\leq\lambda_1^D(\Omega)$ is the Friedlander inequality.

We use these inequalities in \eqref{ko1} to get
$$
\lambda_1(\Omega,\beta)\geq\frac{C_1\beta^2 R^8}{4R_0^6(\beta^2\rho_{\Omega}^2R^2+6\gamma)}
$$
when $\beta\leq\rho_{\Omega}^{-2}$. We note that in this hypothesis, $\beta^2\rho_{\Omega}^2R^2\leq\beta^2\rho_{\Omega}^4\leq 1$. Hence we deduce that for $\beta\leq\rho_{\Omega}^{-2}$
$$
\lambda_1(\Omega,\beta)\geq\frac{C_1\beta^2 R^8}{4R_0^6(1+6\gamma)}.
$$
Analogously, considering \eqref{ko2}, we get that
$$
\lambda_1(\Omega,\beta)\geq\frac{C_1\beta R^6}{4R_0^6(R^2\beta+24\gamma)}
$$
when $\beta\geq\rho_{\Omega}^{-2}$. The proof is concluded by observing that a suitable constant $c$ in the proposition is given by $c=\frac{C_1}{96\gamma}$.
\end{proof}

We recall the following

\begin{defn}A domain $\Omega$ satisfies the $\delta$-interior ball condition if, for any $x\in\partial\Omega$, there exists a ball
of radius $\delta$ tangent to $\partial\Omega$ at $x$ and entirely contained in $\Omega$.
\end{defn}

 For smooth domains, this is equivalent to saying that
the injectivity radius of the normal exponential map is at least $\delta$. Therefore any point of a segment hitting the boundary orthogonally at $p\in\partial\Omega$ minimizes the distance to the boundary up to distance $\delta$ to $p$.

\begin{defn}
Let $\Omega$ be a bounded domain with smooth boundary, and let $\eps>0$. A maximal collection of points $\mathcal P_{\eps}=\{p_1,...,p_n\}$ with the following properties
\begin{itemize}
\item ${\rm dist}(p_j,p_k)\geq\eps$ for all $j\ne k$,
\item ${\rm dist}(p_j,\partial\Omega)\geq\eps$ for all $j$,
\end{itemize}
is called a {\it maximal $\eps$-net}.
\end{defn}

The goal is to produce a general lower bound for the first eigenvalue $\lambda_1(\Omega,\beta)$ of domains $\Omega$ with the $\delta$-interior ball condition depending only on $\delta$ (and $\beta$). To this aim we cover $\Omega$ by star-shaped subdomains $\Omega_i$, and use Proposition \ref{lo_star} to control $\lambda_1(\Omega_i,\beta)$ and then Theorem \ref{ko_imp} to control $\lambda_1(\Omega,\beta)$. In the case of convex domains with the $\delta$-interior ball condition, a suitable covering is proved in  \cite[Lemma 11]{co_sa_agag}. We extend here this last result dropping the convexity assumption.

\begin{lem}\label{agag}
Let $\Omega$ be a bounded domain with smooth boundary and let $\mathcal P_{\eps}=\{p_1,...,p_n\}$ be a maximal $\eps$-net in $\Omega$. Assume that $\Omega$ satisfies the $\delta$-interior ball condition with $\delta\geq \eps$. Then $\Omega$ admits an open covering $\{\Omega_1,...,\Omega_n\}$, $\Omega=\bigcup_{i=1}^n{\Omega_i}$, with the following properties:
\begin{itemize}
\item every $\Omega_i$ is star-shaped with respect to some point $p_i\in\Omega_i$ and has piecewise smooth boundary.
\item For each $i=1,...,n$ one has $B(p_i,\eps/2)\subseteq\Omega_i\subseteq B(p_i,2\eps)$.
\item There exists a universal constant $M\in\mathbb N$ (not depending on $\Omega$) such that a point $x\in\Omega$ can be contained in at most $M$ of the domains $\Omega_i$.
\end{itemize}
\end{lem}
\begin{proof}
We remark that the difficulty of covering $\Omega$ by star-shaped domains $\Omega_i$ lies in the region near the boundary. Indeed, far from the boundary it is easy to find a nice covering. We will proceed in two steps. The family $\{\Omega_i\}$ will be the union of a family $\{\Omega_{i,B}\}$ of domains having a non empty intersection with the boundary $\partial \Omega$ and a family of domains $\{\Omega_{i,I}\}$ which do not intersect the boundary. We will prove the result for $\varepsilon=\delta$; it clearly holds for $0<\varepsilon<\delta$ since the $\delta$-interior ball condition implies  the $\varepsilon$-interior ball condition for all $0<\varepsilon\leq\delta$.

\textbf{Step 1: construction of the domains $\{\Omega_{i,B}\}$.} Let $\partial\Omega_{\delta}$ be the equidistant set to $\partial\Omega$:
$$
\partial \Omega_{\delta}=\{ x\in \Omega: {\rm dist}(x,\partial \Omega)=\delta\}.
$$
By definition, if $x \in \partial \Omega_{\delta}$, then the ball $B(x,\delta) \subset \Omega$ is tangent to $\partial \Omega$ at at least one point.

Let $\mathcal P_1=\{x_1,...,x_n\}$ be a maximal $\frac{\delta}{2}$-net in $\partial \Omega_{\delta}$, that is ${\rm dist}(x_i,x_j)\ge \frac{\delta}{2}$ if $i\not=j$.

\medskip
Then, we define
$$
\Omega_{i,B}=\left\{\bigcup B(x,\delta): x\in \partial \Omega_{\delta} {\rm \ and\ } \ {\rm dist}(x,x_i)\le \frac{\delta}{2}\right\}.
$$
In particular, $x_i \in  B(x,\delta) $ for each $x\in \partial \Omega_{\delta}$ with ${\rm dist}(x,x_i)\le \frac{\delta}{2}$ so that $\Omega_{i,B}$ is star-shaped with respect to the point $x_i$ and
$$
B(x_i,\delta) \subset \Omega_{i,B} \subset B(x_i,2\delta).
$$
We have:
$$
\left\{x\in \Omega: {\rm dist}(x,\partial \Omega)\le \frac{3\delta}{2}\right\}\subset \cup_{i=1}^n \Omega_{i,B}.
$$
To see this, let $y \in \Omega$ with ${\rm dist}(y,\partial \Omega) \le \frac{3\delta}{2}$ and $y'\in \partial \Omega$ be such that ${\rm dist}(y,y')= {\rm dist}(y,\partial \Omega) $.
The line segment determined by $y'$ and $y$ cuts $\partial \Omega_{\delta}$ at a point $x$. By maximality of the net $\mathcal P_1$, there exists $x_i \in \mathcal P_1$ with ${\rm dist}(x,x_i)\le \frac{\delta}{2}$ and $B(x,\delta)\subset \Omega_{i,B}$. As ${\rm dist}(x,y) \le \delta$, we have $y \in \Omega_{i,B}$.

\textbf{Step 2: construction of the domains $\{\Omega_{i,I}\}$.} Let $\mathcal P_2=\{y_1,...,y_m\}$ be a maximal $\frac{\delta}{2}$-net of the domain $\{x\in \Omega: {\rm dist}(x,\partial \Omega)\ge \frac{3\delta}{2}\}$. Then, we choose the domain $\Omega_{i,I}$ to be the ball $B(y_i,\frac{\delta}{2})$. As the intersection with $\partial \Omega$ is empty, $\Omega_{i,I}$ is convex, and star-shaped with respect to $y_i$. By maximality, the domain $\{x\in \Omega: {\rm dist}(x,\partial \Omega)\ge \frac{3\delta}{2}\}$ is covered by $\cup_{i=1}^mB(y_i,\frac{\delta}{2})$. 

\medskip

It follows that the domain $\Omega$ is covered by the union of the domains $\{\Omega_{i,B}\}$ and  $\{\Omega_{i,I}\}$, which we denote by $\{\Omega_i\}_{i=1}^n$. From Step 1 and Step 2 it follows that for each $i=1,...,n$ there exists $p_i\in\Omega_i$ such that $B(p_i,\frac{\delta}{2})\subset\Omega_i\subset B(p_i,2\delta)$, and ${\rm dist}(p_i,p_j)\geq\frac{\delta}{2}$ for all $i\ne j$. It is then easy to deduce that there is a universal constant $M\in\mathbb N$ (not depending on $\Omega$) such that a point can be contained in at most $M$ of these domains (see for example Step 2 of the proof of Theorem 1 in \cite{co_sa_agag}).

\end{proof}

We introduce the following class of domains:
\begin{equation}\label{Adelta}
\mathcal A_{\delta}=\{\Omega\subset\mathbb R^2: \Omega{\rm\ is\ smooth,\  with\ }\delta{\rm-interior\ ball\ condition}\}.
\end{equation}

We are now ready to state the main result of this section
\begin{thm}\label{thm_lower}
There exist a universal constant $C>0$ such that, for all $\Omega\in \mathcal A_{\delta}$
\begin{enumerate}[i)]
\item $\lambda_1(\Omega,\beta)\geq C\beta^2\delta^2$ if $\beta\delta^2\leq 1$;
\item $\lambda_1(\Omega,\beta)\geq C\beta$ if $\beta\delta^2\geq 1$.
\end{enumerate}
\end{thm}
\begin{proof}
We consider first the case $\beta=1$. Let $\Omega\in\mathcal A_{\delta}$ for some $\delta>0$. 

Assume first $\delta\leq 1$. From Lemma \ref{agag} we find that $\Omega$ admits an open covering $\{\Omega_1,...,\Omega_N\}$ with each $\Omega_j$ star-shaped  with respect to some $p_j\in\Omega_j$, and with  $B(p_j,\delta/2)\subseteq\Omega_j\subseteq B(p_j,2\delta)$, $j=1,...,N$. We apply Proposition \ref{lo_star} to each $\Omega_j$, and find that
\begin{equation}\label{lo1}
\lambda_1(\Omega_j,1)\geq c_1\delta^2\,,\ \ \ {\rm if\ }1\leq\rho_{\Omega_j}^{-2}
\end{equation}
and
\begin{equation}\label{lo2}
\lambda_1(\Omega_j,1)\geq c_1\,,\ \ \ {\rm if\ }1 \geq\rho_{\Omega_j}^{-2},
\end{equation}
where $\rho_{\Omega_j}$ is the inradius of $\Omega_j$ and $c_1$ is a universal constant given by $2^{-14}c$. From Theorem \ref{ko_imp}, Lemma \ref{agag}, and the fact that $\delta\leq 1$, we deduce that
\begin{equation}\label{lo3}
\lambda_1(\Omega,1)\geq\frac{c_1}{M}\delta^2.
\end{equation}

Let now $\delta>1$. The fact that $\Omega\in\mathcal A_{\delta}$ implies that $\Omega\in\mathcal A_1$, and with the same arguments above we deduce that 
\begin{equation}\label{lo4}
\lambda_1(\Omega,1)\geq\frac{c_1}{M}.
\end{equation}
In conclusion, from \eqref{lo3} and \eqref{lo4} we get that, for all $\delta>0$,
\begin{equation}\label{lo5}
\lambda_1(\Omega,1)\geq\frac{c_1}{M}\min\{\delta^2,1\}.
\end{equation}
Finally, by applying \eqref{lo5} to the domain $\sqrt{\beta}\Omega$, which has rolling radius $\sqrt{\beta}\delta$, we get that for any $\beta>0$ 
\begin{equation}\label{lo6}
\lambda_1(\Omega,\beta)=\beta\lambda_1(\sqrt{\beta}\Omega,1)\geq\frac{c_1}{M}\min\{\beta^2\delta^2,\beta\}.
\end{equation}
This yields the final result once that we set $C:=c_1/M$.
\end{proof}

\begin{rem}\label{rem:lowerthm}
Let us discuss the bounds of Theorem \ref{thm_lower}. We note that a lower bound behaves like $\beta^2\delta^2$ as $\delta\to 0^+$. This is consistent with the examples of domains with small width (for which $\lambda_1(\Omega,\beta)$ behaves like $\beta^2\delta^2$), see Theorem \ref{small_width}, see also Appendices \ref{sec:thin} and \ref{sub:examples}. Concerning specifically the lower bound $i)$, we note that the behavior is quadratic in $\delta$ and in $\beta$. We have seen in Theorem \ref{up_B_diam} that this quadratic behavior in $\beta$ for small $\beta$ is correct, in particular when $R_{\Omega}\sqrt{\beta}\leq 1$, where $R_{\Omega}$ denotes the circumradius of $\Omega$. As for the linear behavior in $\beta$ of the lower bound $ii)$, this is correct in view of the asymptotic behavior of $\lambda_1(\Omega,\beta)$ as $\beta\to +\infty$, see \eqref{asympt_B}. In any case, we remark that the relevant quantity in our bounds is the behavior of the product $\sqrt{\beta}\delta$ (and not of $\delta$ or $\beta$ alone).
\end{rem}

One may wonder if the hypothesis on the $\delta$-interior ball condition is too restrictive, and if just having a large inradius would imply a large lower bound. This is not the case as we can see in the following example of a convex domain with large inradius and small first eigenvalue.

\begin{ex}\label{ex:tr}
Let $T$ be a triangle with base $\{(x_1,0):x_1\in(-1/2,1/2)\}$ and height of length $L$ (namely, the segment $(0,x_2)$ with $x_2\in(0,L)$). Then, the inradius is uniformly bounded from below and the first eigenvalue vanishes as $L\to+\infty$. To see this, consider the subsets $T'=\{(x_1,x_2)\in T:x_2\in(L-2\sqrt{L},L)\}$ and $T''=\{(x_1,x_2)\in T:x_2\in(L-2\sqrt{L},L-\sqrt{L})\}$. We can build a function $u$ supported on $T'$ with arbitrarily small Rayleigh quotient as follows. Take $u(x_1,x_2)=e^{-\frac{i\beta}{2}x_1x_2}\phi(x_2)$ where $\phi(x_2)\equiv 1$ on $T'\setminus T''$, $\phi(L-2\sqrt{L})=0$ and $\phi(x_2)$ is linear in $T''$. Standard computations (see also Example \ref{conv_small}) show then that $\lambda_1(T,\beta)\leq C L^{-1}$.
\end{ex}

However, Theorem \ref{thm_lower} does not apply to the case of a triangle as in the previous example, since it is not smooth (its rolling radius is $\delta=0$). For non-smooth domains we would rather use Proposition \ref{lo_star}, which gives a good lower bound for the domain in Example \ref{ex:tr}.

\section{Semiclassical estimates for averages of eigenvalues}\label{sec:sums}

In this section we prove asymptotically sharp lower bounds for the first Riesz mean of magnetic eigenvalues (upper bounds for averages) in the spirit of Kr\"oger \cite{Kro}. This will imply upper bounds on single eigenvalues. We recall that the first Riesz mean $R_1$ of magnetic eigenvalues is defined by $R_1(z)=\sum_{j=1}^{\infty}(z-\lambda_j)_+$, where $a_+=\max\{0,a\}$. Through all this section we shall drop the dependance of $\lambda_j(\Omega,\beta)$ on $\beta,\Omega$ and simply write $\lambda_j$. We also denote by $[a]$ the integer part of $a\in\mathbb R$ and by $\psi(a):=a-[a]-\frac{1}{2}$, $a\in\mathbb R$, the fluctuation function.
\begin{thm}\label{main_averages}
Let $\Omega$ be a bounded domain of $\mathbb R^2$.  For all $z\geq 0$ we have
\begin{equation}\label{main_R1}
R_1(z)\geq \frac{|\Omega|}{8\pi}z^2-\frac{\beta^2|\Omega|}{2\pi}\psi^2\left(\frac{z}{2\beta}+\frac{1}{2}\right).
\end{equation}
Equivalently, for all $k\in\mathbb N$, $k\geq 1$ we have
\begin{equation}\label{up_2}
\frac{1}{k}\sum_{j=1}^k\lambda_j\leq\frac{2\pi k}{|\Omega|}+R\left(\frac{2\pi k}{\beta|\Omega|}\right),
\end{equation}
where
\begin{equation}\label{up_2_rem}
R(X)=\frac{\beta}{X}(X-[X])([X]-X+1)
\end{equation}
In particular, for all $X\geq 0$
\begin{equation}\label{up_2_rem_est}
0\leq R(X)\leq\frac{\beta}{4X},
\end{equation}
and $R(X)=\beta(1-X)$ for $0\leq X\leq 1$. Therefore, for $k\leq\frac{\beta|\Omega|}{2\pi}$ we have
\begin{equation}\label{up_1}
\frac{1}{k}\sum_{j=1}^k\lambda_j\leq\beta.
\end{equation}
\end{thm}
Before proving the theorem, we state a few remarks and consequences. We observe that Theorem \ref{main_averages} implies bounds on single eigenvalues, as in \cite{Kro}.

\begin{cor}\label{single_bounds}
For any $k\in\mathbb N$ we have
\begin{equation}\label{single_0}
\lambda_{k+1}\leq\frac{4\pi}{|\Omega|}\left(k+\sqrt{k^2+\frac{\beta^2|\Omega|^2}{4\pi^2}\psi^2\left(\frac{\lambda_{k+1}}{2\beta}+\frac{1}{2}\right)}\right)
\end{equation}
In particular, since $0\leq\psi^2(a)\leq\frac{1}{4}$ for all $a\in\mathbb R$,
\begin{equation}\label{single_1}
\lambda_{k+1}\leq\frac{8\pi k}{|\Omega|}+\beta.
\end{equation}
\end{cor}
\begin{proof}
For $k=0$ we have already proved that $\lambda_1<\beta$ in Theorem \ref{up_B_diam}. We consider \eqref{main_R1} with $z=\lambda_{k+1}$. Then for the left-hand side of \eqref{main_R1} we have
$$
R_1(\lambda_{k+1})=\sum_{j=1}^k(\lambda_{k+1}-\lambda_j)=k\lambda_{k+1}-\sum_{j=1}^k\lambda_j\leq k\lambda_{k+1}.
$$
Hence we obtain the inequality
$$
k\lambda_{k+1}\geq \frac{\lambda_{k+1}^2|\Omega|}{8\pi}-\frac{\beta^2|\Omega|}{2\pi}\psi^2\left(\frac{\lambda_{k+1}}{2\beta}+\frac{1}{2}\right)
$$
which yields \eqref{single_0}. Since $0\leq\psi^2(a)\leq\frac{1}{4}$ for all $a\in\mathbb R$, and using $\sqrt{a^2+b^2}\leq a+b$ for all $a,b\geq 0$, we obtain \eqref{single_1}.
\end{proof}


\begin{rem}
Let us compare the bounds given by Theorem \ref{main_averages} with the corresponding bounds for the Neumann Laplacian in two dimensions proved in \cite{Kro}:
\begin{equation}\label{bounds_kro}
\frac{1}{k}\sum_{j=1}^k\lambda_j^N\leq\frac{2\pi k}{|\Omega|}.
\end{equation}
Here by $\lambda_j^N$ we denote the Neumann eigenvalues of the Laplacian on $\Omega$.
Clearly, in our situation, bounds of the form \eqref{bounds_kro} cannot hold for any $k$ and any value of $|\Omega|$: the inequality is clearly violated when $k=1$ and $|\Omega|\to+\infty$. Hence it is natural to distinguish the regime $|\Omega|\geq\frac{2\pi k}{\beta}$ and $|\Omega|\leq\frac{2\pi k}{\beta}$. Also, the appearing of oscillations in the remainders of the estimates of Theorem \ref{main_averages} seems to be natural for this operator (see Appendices \ref{sec:thin} and \ref{sec:disk}). The semiclassical estimates of Theorem \ref{main_averages} should be compared with those for the magnetic Dirichlet Laplacian proved in \cite{erdos}. In particular, for magnetic Dirichlet eigenvalues a lower bound on eigenvalues averages is given by the Weyl term, as for the Laplacian.
\end{rem}

Another Corollary of Theorem \ref{main_averages} is the following lower bound on the trace of the magnetic heat kernel, which is asymptotically sharp as $t\to 0^+$.
\begin{cor}\label{heat_kernel}
For all $t>0$ we have
$$
\sum_{j=1}^{\infty}e^{-\lambda_jt}\geq\frac{\beta|\Omega|}{4\pi \sinh(\beta t)}
$$
\end{cor}
\begin{proof}
The inequality follows by Laplace transforming inequality \eqref{main_R1}.
\end{proof}


The proof of Theorem \ref{main_averages} relies on the so-called {\it averaged variational principle} of Harrell-Stubbe \cite{HaSt14}, which is an efficient way of recovering Kr\"oger's result \cite{Kro}, and can be easily applied in various situations. We recall it here for the reader's convenience.

\begin{thm}\label{thm:AVP}
Let $H$ be a self-adjoint operator in a Hilbert space $(\mathcal{H}, \langle\cdot,\cdot,\rangle_\mathcal{H})$ with discrete spectrum, made of eigenvalues denoted by
\[
\omega_1 \leq \omega_2 \leq \cdots \leq \omega_j \leq \cdots
\]
 with corresponding orthonormalized eigenvectors $\{g_j\}_{j \in \mathbb{N}\setminus\{0\}}$.
The closed quadratic form corresponding to $H$ is denoted $Q(\varphi,\varphi)$ for any $\varphi$ in
the quadratic form domain $\mathcal{Q}(H) \subset \mathcal{H}$. Let $f_p \in \mathcal{Q}(H)$ be a family
of vectors indexed by a variable $p$ ranging over a measure space $(\mathfrak{M}, \Sigma, \sigma)$. Suppose  that $\mathfrak{M}_0$ is a measurable subset of $\mathfrak{M}$.
Then for any $z \in \mathbb{R}$,
\begin{equation}\label{AVPrm}
\sum_{j \in \mathbb{N}}(z-\omega_j)_+ \int_{\mathfrak{M}}\left| \langle g_j, f_p \rangle_{\mathcal{H}}
\right|^2 \, d\sigma \geq \int_{\mathfrak{M}_0} \left( z\|f_p \|_{\mathcal{H}}^2-Q(f_p,f_p) \right)\,d\sigma,
\end{equation}
provided that the integrals converge. Here $a_+$ denotes the positive part of a real number $a$.
\end{thm}

We state the following Lemma which contains some known facts on eigenfunctions of $\Delta_A$ on $\mathbb R^2$.

\begin{lem}\label{laguerrelem}
Let $u_{n,l}\in C^{\infty}(\mathbb R^2)$ be defined in polar coordinates $(r,t)$ by
$$
u_{n,l}(r,t):=e^{-\frac{\beta r^2}{4}}r^n L_l^n\left(\frac{r^2 \beta}{2}\right)e^{int},
$$
where by $L_l^n(y)$ we denote the associated Laguerre polynomial, namely
\begin{equation}\label{laguerre}
L_l^n(y)=\sum_{i=0}^l(-1)^i\binom{l+n}{l-i}\frac{y^i}{i!}=y^{-n}\frac{e^y}{l!}\frac{d^l}{dy^l}(y^{l+n}e^{-y})
\end{equation}
with $l,n\in\mathbb N$ (in particular, $L^n_0(y)=1$). Then 

\begin{enumerate}[i)]
\item $\Delta_A u_{n,l}=\beta (1+2l) u_{n,l}$ on $\mathbb R^2$: the functions $u_{n,l}$ are eigenfunctions of $\Delta_A$ on $\mathbb R^2$ with eigenvalue $\beta(1+2l)$. Each eigenspace has infinite dimension. 
\item $|\nabla^Au_{n,l}|^2=\beta(1+2l)|u_{n,l}|^2-\frac{1}{2}\Delta|u_{n,l}|^2$.
\item $\int_{\mathbb R^2}u_{n,l}\overline {u_{m,k}}=0$ if $m\ne n$ or $l\ne k$.
\item $\int_{\mathbb R^2}|u_{n,l}|^2=\pi\left(\frac{2}{\beta}\right)^{n+1}\frac{(l+n)!}{l!}=:c_{n,l}^2$.
\item Let $v_{n,l}:=\frac{u_{n,l}}{c_{n,l}}$; then $\int_{\mathbb R^2}v_{n,l}\overline{v_{m,k}}=\delta_{mn}\delta_{lk}$, hence $\{v_{n,l}\}_{n,l\in\mathbb N}$ is an orthonormal system in $L^2(\mathbb R^2)$.
\end{enumerate}
\end{lem}
The proof follows from standard calculus, using the expression of $\Delta^A$ in polar coordinates \eqref{laplacian_polar} (see also \cite{bauman}).

The next lemma establishes a basic inequality for $R_1(z)$ which is the cornerstone of the proof of Theorem \ref{main_averages}.

\begin{lem} For all $z\geq 0$ we have
\begin{equation}\label{R1_2}
\sum_{j=1}^{\infty}(z-\lambda_j)_+
\geq\frac{\beta|\Omega|}{2\pi}\sum_{l=1}^{\infty}(z-\beta(2l-1))_+.
\end{equation}
\end{lem}
\begin{proof}
We apply the {\it averaged variational principle} \eqref{AVPrm} with  $\mathcal H=L^2(\Omega)$,  $H=\Delta_A$, $Q(f,f)=\int_{\Omega}|\nabla^Af|^2$, $\mathcal Q(H)=H^1(\Omega)$, $\omega_j=\lambda_j$, $g_j=u_j$, where $u_j$ are the the $L^2(\Omega)$-normalized eigenfunctions associated with the eigenvalues $\lambda_j$ of $\Delta_A$ on $\Omega$,
 $\mathfrak M=\mathbb N\times\mathbb N$ with the counting measure $\sigma$, and $\mathfrak M_0=\mathbb N\times \{0,...,L\}$ for some $L\in\mathbb N$, and $f_p=v_{n,l}$. Then \eqref{AVPrm} reads
\begin{equation}
\sum_{j=1}^{\infty}\left[(z-\lambda_j)_+\sum_{l=0}^{\infty}\sum_{n=0}^{\infty}\left|\int_{\Omega} v_{n,l}\overline {u_j}\right|^2\right]\geq\sum_{l=0}^L\sum_{n=0}^{\infty}\int_{\Omega}\left(z|v_{n,l}|^2 - |\nabla^Av_{n,l}|^2\right).
\end{equation}
Now, since $\{v_{n,l}\}_{n,l\in\mathbb N}$ is an orthonormal family in $L^2(\mathbb R^2)$, we have that 
$$
\sum_{l=0}^{\infty}\sum_{n=0}^{\infty}\left|\int_{\Omega} v_{n,l}\overline {u_j}\right|^2\leq\int_{\Omega}|u_j|^2=1
$$
hence, by Lemma \ref{laguerrelem}, ii)
\begin{multline}\label{main_average}
\sum_{j=1}^{\infty}(z-\lambda_j)_+\geq\int_{\Omega}\left(z\sum_{l=0}^L\sum_{n=0}^{\infty}|v_{n,l}|^2 - \sum_{l=0}^L\sum_{n=0}^{\infty}|\nabla^Av_{n,l}|^2\right)\\
=\int_{\Omega}\left(\sum_{l=0}^L\sum_{n=0}^{\infty}(z-\beta(2l+1))|v_{l,n}|^2 + \frac{1}{2}\sum_{l=0}^L\sum_{n=0}^{\infty}\Delta|v_{n,l}|^2\right)\\
=
\sum_{l=0}^L(z-\beta(2l+1))\sum_{n=0}^{\infty}\int_{\Omega}|v_{l,n}|^2 + \frac{1}{2}\sum_{l=0}^L\sum_{n=0}^{\infty}\int_{\Omega}\Delta|v_{n,l}|^2.
\end{multline}
We will prove in Lemma \ref{sum_laguerre} here below that
$$
\sum_{n=0}^{\infty}|v_{n,l}|^2=\frac{\beta}{2\pi}\left(1+e^{-\frac{\beta |x|^2}{2}}P_l\left(\frac{\beta |x|^2}{2}\right)\right),
$$
where $P_l(y)$ is a polynomial of degree $2l-1$ in the variable $y$. The convergence is uniform on any compact set. In particular, $\sum_{n=0}^{\infty}\Delta|v_{n,l}|^2=\frac{\beta}{2\pi}\Delta\left(e^{-\frac{\beta |x|^2}{2}}P_l(\beta |x|^2/2)\right)$. 

Then, from \eqref{main_average} we get that, for all $L\in\mathbb N$ and $z\geq 0$,
\begin{multline}\label{R1_0}
\sum_{j=1}^{\infty}(z-\lambda_j)_+
\geq\int_{\Omega}\frac{\beta}{2\pi}\sum_{l=0}^L\left(1+e^{-\frac{\beta |x|^2}{2}}P_l\left(\frac{\beta |x|^2}{2}\right)\right)(z-\beta(2l+1))\\+\frac{\beta}{4\pi}\sum_{l=0}^L\int_{\Omega}\Delta\left(e^{-\frac{\beta|x|^2}{2}}P_l\left(\frac{\beta |x|^2}{2}\right)\right).
\end{multline}
Inequality \eqref{R1_0} holds for any fixed $L$, and it is clearly valid if we replace $|x|$ by $|x+x_0|$, $x_0\in\mathbb R^2$ (this amounts to choosing $v_{n,l}(r,t)$ where $(r,t)$ are polar coordinates centered at $x_0$). Then, for any $L\in\mathbb N$, taking $|x_0|\to+\infty$, we deduce that the last term of \eqref{R1_0} goes to $0$, and hence
\begin{equation}\label{R1_1}
\sum_{j=1}^{\infty}(z-\lambda_j)_+
\geq\frac{\beta|\Omega|}{2\pi}\sum_{l=0}^L(z-\beta(2l+1))=\frac{\beta|\Omega|}{2\pi}\sum_{l=1}^{L+1}(z-\beta(2l-1)).
\end{equation}
This implies immediately \eqref{R1_2}.
\end{proof}

We are now ready to prove Theorem \ref{main_averages}.

\begin{proof}[Proof of Theorem \ref{main_averages}]

Inequality \eqref{R1_2} is the cornerstone of this proof. The statements of Theorem \ref{main_averages} are consequences of this inequality.
 We prove first \eqref{main_R1}. We consider the right-hand side of \eqref{R1_2} and re-write it as
$$
\frac{\beta^2|\Omega|}{2\pi}\sum_{l=1}^{\infty}(\frac{z}{\beta}-(2l-1))_+=\frac{\beta^2|\Omega|}{2\pi}\sum_{l=1}^{\infty}((2w-1)-(2l-1))_+
$$
where $w=\frac{z}{2\beta}+\frac{1}{2}$. Then, some algebraic manipulations yield
\begin{multline*}
\frac{\beta^2|\Omega|}{2\pi}\sum_{l=1}^{\infty}((2w-1)-(2l-1))_+=\frac{\beta^2|\Omega|}{2\pi}\sum_{l=1}^{[w]}((2w-1)-(2l-1))=\frac{\beta^2|\Omega|}{2\pi}[w](2w-[w]-1)\\
=\frac{\beta^2|\Omega|}{2\pi}\left(w-\frac{1}{2}-\psi(w)\right)\left(w-\frac{1}{2}+\psi(w)\right)=\frac{\beta^2|\Omega|}{2\pi}\left(w-\frac{1}{2}\right)^2-\frac{\beta^2|\Omega|}{2\pi}\psi^2(w)\\
=\frac{\beta^2|\Omega|}{2\pi}\cdot\frac{z^2}{4\beta^2}-\frac{\beta^2|\Omega|}{2\pi}\psi^2\left(\frac{z}{2\beta}+\frac{1}{2}\right).
\end{multline*}
This proves \eqref{main_R1}.

\medskip
Before proving \eqref{up_2}, we note that, choosing $z=\lambda_1$, inequality \eqref{R1_2} reads
$$
0\leq \frac{\beta|\Omega|}{2\pi}\sum_{l=1}^{\infty}(\lambda_1-\beta(2l-1))_+\leq 0
$$
which implies $\lambda_1\leq\beta(2l-1)$ for all $l\geq 1$, and in particular, $\lambda_1\leq\beta$. This is an alternative way of recovering \eqref{TA} (however the inequality is not strict). 

\medskip

Now we prove \eqref{up_2}. Consider now the two functions
$$
f(z)=\sum_{j=1}^{\infty}(z-\lambda_j)_+
$$
and
$$
g(z)=\frac{\beta|\Omega|}{2\pi}\sum_{l=1}^{\infty}(z-\beta(2l-1))_+
$$
These two functions are convex. Let us define, for any $w\geq 0$, the functions
$$
\mathcal L[f](w):=\sup_{z\geq 0}(zw-f(z))\,,\ \ \ \mathcal L[g](w):=\sup_{z\geq 0}(zw-g(z)).
$$

These two functions are the Legendre transforms of $f$ and $g$. Since $f,g$ are convex, we have, for all $w\geq 0$, that $f(z)\geq g(z)\iff$ $\mathcal L[f](w)\leq\mathcal L[g](w)$.

A standard computation shows that
$$
\mathcal L[f](w)=\lambda_{[w]+1}(w-[w])+\sum_{j=1}^{[w]}\lambda_j
$$
and
\begin{multline*}
\mathcal L[g](w)=\frac{\beta^2|\Omega|}{2\pi}\left(2\left[\frac{2\pi w}{\beta|\Omega|}\right]+1\right)\left(\frac{2\pi w}{\beta|\Omega|}-\left[\frac{2\pi w}{\beta|\Omega|}\right]\right)+\sum_{j=1}^{\left[\frac{2\pi w}{\beta|\Omega|}\right]}\frac{\beta^2|\Omega|(2j-1)}{2\pi}\\
=\frac{\beta^2|\Omega|}{2\pi}\left(2\left[\frac{2\pi w}{\beta|\Omega|}\right]+1\right)\left(\frac{2\pi w}{\beta|\Omega|}-\left[\frac{2\pi w}{\beta|\Omega|}\right]\right)+\frac{\beta^2|\Omega|}{2\pi}\left[\frac{2\pi w}{\beta|\Omega|}\right]^2
\end{multline*}
We choose now $w=k$, so that the inequality $\mathcal L[f](k)\leq L[g](k)$ reads
\begin{equation}\label{sums0}
\sum_{j=1}^{k}\lambda_j\leq \frac{\beta^2|\Omega|}{2\pi}\left(2\left[\frac{2\pi k}{\beta|\Omega|}\right]+1\right)\left(\frac{2\pi k}{\beta|\Omega|}-\left[\frac{2\pi k}{\beta|\Omega|}\right]\right)+\frac{\beta^2|\Omega|}{2\pi}\left[\frac{2\pi k}{\beta|\Omega|}\right]^2
\end{equation}
Setting $X=\frac{2\pi k}{\beta|\Omega|}$ in \eqref{sums0}, we see that
\begin{multline*}
\sum_{j=1}^k\lambda_j\leq\frac{\beta^2|\Omega|}{2\pi}(2[X]+1)(X-[X])+\frac{\beta^2|\Omega|}{2\pi}[X]^2\\
=\frac{\beta^2|\Omega|}{2\pi}(2[X]+1)(X-[X])+\frac{\beta^2|\Omega|}{2\pi}(X^2+[X]^2-X^2)\\
=\frac{\beta^2|\Omega|}{2\pi}X^2+\frac{\beta^2|\Omega|}{2\pi}(2[X]+1)(X-[X])+\frac{\beta^2|\Omega|}{2\pi}([X]^2-X^2)\\
=\frac{2\pi k^2}{|\Omega|}+\frac{\beta^2|\Omega|}{2\pi}(2[X]+1)(X-[X])+\frac{\beta^2|\Omega|}{2\pi}([X]^2-X^2),
\end{multline*}
where we have used the fact that $\frac{\beta^2|\Omega|}{2\pi}X^2=\frac{2\pi k^2}{|\Omega|}$. Dividing both sides by $k$ we find
\begin{equation}\label{sums2}
\frac{1}{k}\sum_{j=1}^k\lambda_j\leq \frac{2\pi k}{|\Omega|}+\frac{\beta}{X}(X-[X])([X]-X+1)
\end{equation}
where we recall that $X=\frac{2\pi k}{\beta|\Omega|}$. This proves \eqref{up_2}.

\medskip

We give an estimate for the reminder function
$$
R(X)=\frac{\beta}{X}(X-[X])([X]-X+1).
$$
First of all, note that the function
$$
G(X)=(X-[X])([X]-X+1)
$$
is such that $G(n)=0$ for all $n\in\mathbb N$, $G(x)\geq 0$, and on each interval $(n,n+1)$ has a unique maximum which is $\frac{1}{4}$ and is reached when $X=n+\frac{1}{2}$. Hence $0\leq G(X)\leq\frac{1}{4}$. Therefore
$$
0\leq R(X)\leq\frac{\beta}{4X}.
$$
which is \eqref{up_2_rem_est}. 
If $0\leq X\leq 1$ one immediately checks that $R(X)=\beta(1-X)$. If $k\leq\frac{\beta|\Omega|}{2\pi}$ then $0\leq X\leq 1$ and from \eqref{sums2} we immediately get 
$$
\frac{1}{k}\sum_{j=1}^k\lambda_j\leq  \beta.
$$
which is \eqref{up_1}. The proof of Theorem \ref{main_averages} is concluded.

\end{proof}

We prove the following lemma on sums of eigenfunctions of $\Delta_A$ on $\mathbb R^2$.

\begin{lem}\label{sum_laguerre}
We have
$$
\sum_{n=0}^{\infty}|v_{n,l}|^2=\frac{\beta}{2\pi}\left(1+e^{-\frac{\beta |x|^2}{2}}P_l\left(\frac{\beta |x|^2}{2}\right)\right)
$$
where $P_l(y)$ is a polynomial of degree $2l-1$ in the variable $y$. If $l=0$, then $P_0(y)=0$.
\end{lem}
\begin{proof}
We note that
$$
|v_{n,l}|^2=\frac{\beta}{2\pi}e^{-y}\frac{l!}{(l+n)!}y^nL_l^n(y)^2
$$
with $y=\frac{\beta |x|^2}{2}$, so that
\begin{equation}\label{lag0}
\sum_{n=0}^{\infty}|v_{n,l}|^2=\frac{\beta}{2\pi}e^{-y}\sum_{n=0}^{\infty}\frac{l!}{(l+n)!}y^nL_l^n(y)^2
\end{equation}
Therefore, we need to study
\begin{equation}\label{prod0}
\sum_{n=0}^{\infty}\frac{l!}{(l+n)!}y^nL_l^n(y)^2.
\end{equation}
We use \eqref{laguerre} to expand one factor $L_l^n(y)$ in \eqref{prod0}
$$
\sum_{n=0}^{\infty}\frac{l!}{(l+n)!}y^nL_l^n(y)^2=\sum_{n=0}^{\infty}\frac{l!}{(l+n)!}y^nL_l^n(y)\sum_{i=0}^l(-1)^i\binom{l+n}{l-i}\frac{y^i}{i!}
$$
and change the order of summation:
$$
\sum_{n=0}^{\infty}\frac{l!}{(l+n)!}y^nL_l^n(y)^2=\sum_{i=0}^l \frac{(-1)^iy^i}{(l-i)!i!}\sum_{n=0}^{\infty}\frac{y^nL_l^n(y)(l+n)!l!}{(n+i)!(l+n)!}
$$
Using the second identity (Rodrigues formula) in \eqref{laguerre} we get
\begin{multline}\label{lag1}
\sum_{n=0}^{\infty}\frac{l!}{(l+n)!}y^nL_l^n(y)^2=\sum_{i=0}^l \frac{(-1)^iy^i}{(l-i)!i!}\sum_{n=0}^{\infty}\frac{e^y}{(n+i)!}\frac{d^l}{dy^l}(y^{l+n}e^{-y})\\
=e^y\sum_{i=0}^l\frac{(-1)^iy^i}{(l-i)!i!}\frac{d^l}{dy^l}\left(e^{-y}y^{l-i}\sum_{n=0}^{\infty}\frac{y^{n+i}}{(n+i)!}\right)\\=e^y\sum_{i=0}^l\frac{(-1)^iy^i}{(l-i)!i!}\frac{d^l}{dy^l}\left(e^{-y}y^{l-i}\left(e^y-\sum_{j=0}^{i-1}\frac{y^j}{j!}\right)\right)\\
=e^{y}\sum_{i=0}^l\binom{l}{i}\frac{(-1)^iy^i}{l!}\frac{d^l}{dy^l}\left(y^{l-i}-y^{l-i}e^{-y}\sum_{j=0}^{i-1}\frac{y^j}{j!}\right)\\
=e^y+e^{y}\sum_{i=1}^l\binom{l}{i}\frac{(-1)^iy^i}{l!}\frac{d^l}{dy^l}\left(y^{l-i}-y^{l-i}e^{-y}\sum_{j=0}^{i-1}\frac{y^j}{j!}\right).
\end{multline}
The proof follows now inserting \eqref{lag1} in \eqref{lag0}, and observing that the second summand in the last line of \eqref{lag1} is just a polynomial of degree $2l-1$.
\end{proof}

\appendix

\section{Eigenvalues of embedded curves and  thin tubular neighborhoods}\label{sec:thin}

In this section we consider the magnetic Laplacian on embedded curves. Throughout this section, by $\Gamma$ we denote a simple, closed curve, which is the boundary of some simply connected domain $\Omega$ in $\mathbb R^2$ (namely, $\Gamma=\partial\Omega$). 

 As potential one-form, we consider the restriction $\hat A$ of the standard magnetic potential $A=\frac{\beta}{2}(-x_2dx_1+x_1dx_2)$ to $\Gamma$ and study the resulting one-dimensional magnetic operator. 

That is, if $f:\Gamma\to\mathbb R^2$ is the embedding, we take the pull-back $\hat A=f^{\star}A$. It should be noted that $\hat A$ is closed for dimensional reasons (i.e., $d\hat A=0$), hence it generates a vanishing magnetic field on $\Gamma$. We denote by $\lambda_j(\Gamma,\hat A)$ the corresponding eigenvalues, which can be explicitly computed.

\begin{thm}\label{thm1d}
Let $\Gamma$ be an embedded curve, which is the boundary of a simply connected domain $\Omega\subset\mathbb R^2$, and consider the magnetic Laplacian associated with the potential $\hat A$ as above. Its spectrum is then
given by the collection
$$
\frac{4\pi^2}{|\Gamma|^2}\left(n-\frac{\beta|\Omega|}{2\pi}\right)^2, \quad n\in\mathbb Z.
$$
In particular
\begin{equation}\label{embeddedcurves}
\lambda_1(\Gamma,\hat A)=\frac{4\pi^2}{|\Gamma|^2}\min_{n\in\mathbb Z}\Big(n-\frac{\beta|\Omega|}{2\pi}\Big)^2
\end{equation}
hence $\lambda_1(\Gamma,\hat A)=0$ if and only if $|\Omega|=\frac{2n\pi}{\beta}$ for some $n\in\mathbb N$.
\end{thm}

\begin{proof}
Note that, being a compact one-dimensional Riemannian manifold, $\Gamma$ is isometric to the circle with the same length; from \cite[Proposition 7]{CS1}, we know that the spectrum is given by:
$$
\lambda_n(\Gamma,\hat A)=\dfrac{4\pi^2}{\abs{\Gamma}^2}(n-\Phi^{A})^2, \quad n\in\mathbb Z
$$
where $\Phi^{A}=\dfrac{1}{2\pi}\int_{\Gamma}\hat A$ is the flux of $\hat A$ around $\Gamma$ oriented counter-clockwise (however the spectrum does not depend on the orientation). We compute the flux knowing that $\Gamma=\partial\Omega$ and get, by the Stokes formula:
$$
\Phi^{A}=\dfrac{1}{2\pi}\int_{\Omega}dA=\dfrac{\beta\abs{\Omega}}{2\pi}.
$$
The conclusion follows. 
\end{proof}

As a corollary, the classical isoperimetric inequality implies the following fact which, by abuse of language, can be interpreted as a ``reverse Faber-Krahn inequality'' for the first magnetic eigenvalue of the boundary of simply connected domains.

\begin{thm}\label{reverse_curves}
Let $\Omega\subset\mathbb R^2$ be a smooth simply connected domain with boundary $\Gamma$, and let $\Omega^*$ be a disk with $|\Omega|=|\Omega^*|$ and boundary $\Gamma^*$. Then
$$
\lambda_1(\Gamma,\hat A)\leq\lambda_1(\Gamma^*,\hat A).
$$
If $|\Omega|\ne\frac{2n\pi}{\beta}$ for all $n\in\mathbb N$, then equality holds if and only if $\Omega=\Omega^*$. In particular, we have
\begin{equation}\label{B4}
\lambda_1(\Gamma,\hat A)\leq\frac{\beta}{4}
\end{equation}
with equality if and only if $\Omega$ is a disk of radius $R=\frac{1}{\sqrt\beta}$.
\end{thm}
 \begin{proof} The first assertion is an immediate consequence of \eqref{embeddedcurves} and the isoperimetric inequality. We prove the second assertion.
Assume first that $\frac{\beta\abs{\Omega}}{2\pi}\leq \frac 12$ so that we have $\min_{n\in\mathbb Z}\left(n-\frac{\beta|\Omega|}{2\pi}\right)^2=\left(\frac{\beta\abs{\Omega}}{2\pi}\right)^2\leq
\frac 12\frac{\beta\abs{\Omega}}{2\pi}$ hence
$$
\lambda_1(\Gamma,\hat A)\leq \dfrac{4\pi^2}{\abs{\Gamma}^2}\cdot\frac 12\frac{\beta\abs{\Omega}}{2\pi}\leq\dfrac{\beta}{4}
$$
by the isoperimetric inequality $\abs{\Gamma}^2\geq 4\pi\abs{\Omega}$. Note that the equality holds if and only if $\Omega$ is a disk of area $\frac{\pi}{\beta}$. Then, we assume $\frac{\beta\abs{\Omega}}{2\pi}> \frac 12$, and observe that $\min_{n\in\mathbb Z}\left(n-\frac{\beta|\Omega|}{2\pi}\right)^2\leq \frac 14$. It follows that
$$
\lambda_1(\Gamma,\hat A)\leq \dfrac{4\pi^2}{\abs{\Gamma}^2}\cdot \frac 14\leq\dfrac{\pi}{4\abs{\Omega}}<\frac{\beta}{4}.
$$
Finally, one checks easily that for a disk of radius $R=\frac{1}{\sqrt{\beta}}$ we have equality in \eqref{B4}.
The proof is complete.
\end{proof}

\begin{rem}
Note that, in the case of a circle $\Gamma_R$ of radius $R$, for a fixed $\beta$, we always have a sequence of radii $R_n$ such that $\lambda_1(\Gamma_{R_n},\hat A)=0$. This amounts to $R_n=\sqrt{\frac{2n}{\beta}}$, $n\in\mathbb N$.

Note also, that for these values of $R_n$, we have that $\beta$ is an eigenvalue of the magnetic Laplacian on $B_{R_n}$ (see Appendix \ref{sec:disk}).

Moreover, $\lambda_1(\Gamma_R,\hat A)\to 0$ as $R\rightarrow+\infty$, but the convergence is not monotonic. There is a subsequence, as we have said, where it is equal to zero. See Figure \ref{fig1d}. We note again the oscillating behavior of the first eigenvalue as a function of the radius. We observed an analogous behavior in the remainder of the lower bound for $R_1(z)$ in Theorem \ref{main_averages}. Also, an oscillating behavior is evident numerically for the magnetic eigenvalues of disks (see Figure \ref{F1}).

\end{rem}

\begin{figure}
\includegraphics[width=0.7\textwidth]{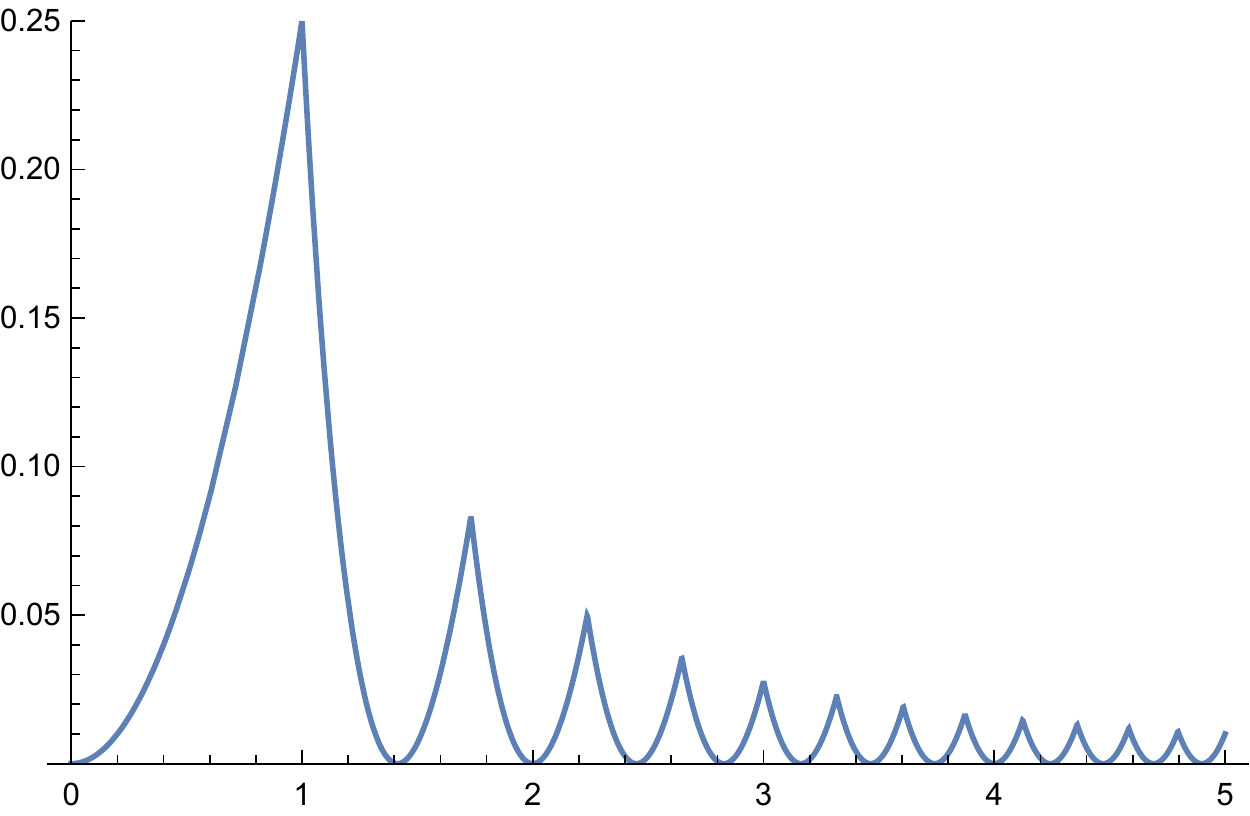}
\caption{First eigenvalue on a circle of radius $R$ as a function of $R$, for $\beta=1$.}
\label{fig1d}
\end{figure}

\section{Eigenvalues of the disk}\label{sec:disk}

In this section we consider the eigenvalues of the disk $B_R:=B(0,R)$. We shall denote them by $\lambda_j(B_R,\beta)$. Writing $\Delta_Au=\lambda u$ in $B_R$ in polar coordinates $(r,t)$ (see \eqref{laplacian_polar}), under the ansatz $u=v(r)e^{in t}$, $n\in\mathbb Z$, we see that $v$ solves
\begin{equation}\label{SLdisk}
\begin{cases}
-v''(r)-\frac{v'(r)}{r}+\left(\frac{\beta r}{2}-\frac{n}{r}\right)^2v(r)=\lambda v(r)\\
v'(R)=0.
\end{cases}
\end{equation}
A bounded solution to the differential equation in \eqref{SLdisk} is given by
\begin{equation}\label{ef_disk}
v_n(r)=e^{-\frac{\beta r^2}{4}}r^n L_{\frac{1}{2}\left(\frac{\lambda}{\beta}-1\right)}^n\left(\frac{r^2\beta}{2}\right),
\end{equation}
where $L_{\gamma}^a(x)$ denotes the generalized Laguerre function (see e.g., \cite[\S 22]{abram} for precise definitions and properties). The eigenvalues are  determined by imposing $v_n'(R)=0$.
For each $n\in\mathbb Z$ we have a sequence
$$
0<\lambda_1(n,\beta,R)<\lambda_2(n,\beta,R)\leq\cdots\nearrow+\infty.
$$
Clearly,
$$
\lambda_1(B_R,\beta)=\min_{n\in\mathbb Z}\lambda_1(n,\beta,R).
$$
This minimum depends on $\beta$ and $R$, but it is rather complicated to identify which $n\in\mathbb Z$ realizes it, for given $\beta, R$. It is easy to see that the minimum, in any case, is achieved for some $n\in\mathbb N$. The analytic branches (i.e., $\lambda_1(n,\beta,R)$) show a very intriguing behavior. A first numerical study can be found e.g., in \cite{SJ}. Due to the rescaling properties of $\lambda_1(B_R,\beta)$, it is not restrictive to fix $\beta=1$ and study the behavior of the first eigenvalue as function of $R$.

Note that the analogous study for the magnetic Laplacian with Dirichlet conditions is simpler (even if the computations are not trivial), and a complete picture is available (see \cite {son}). In particular, the first eigenfunction is always radial, i.e., $n=0$. 

The peculiar behavior of the eigenvalue $\lambda_1(B_R,\beta)$ as a function of $R$ is well illustrated in Figure \ref{F1}. Recall that the eigenvalues of $B_R$ are implicitly characterized by $v_n'(R)=0$, where $v_n$ is defined in \eqref{ef_disk}. In Figure \ref{F1} we have plotted the zero level sets of the function
$$
F(R,\lambda)=v_n'(R)=\frac{d}{dr}\left(e^{-\frac{\beta r^2}{4}}r^n L_{\frac{1}{2}\left(\frac{\lambda}{\beta}-1\right)}^n\left(\frac{r^2\beta}{2}\right)\right)_{|_{r=R}},
$$
for the choice $\beta=1$, in the region $(R,\lambda)\in[0,6]\times[0,1]$ for $n=0,...,10$. The first eigenvalue is the minimum of all the analytic branches of eigenvalues $\lambda_1(n,1,R)$ for $n=0,...,10$.

\begin{figure}
\includegraphics[width=0.8\textwidth]{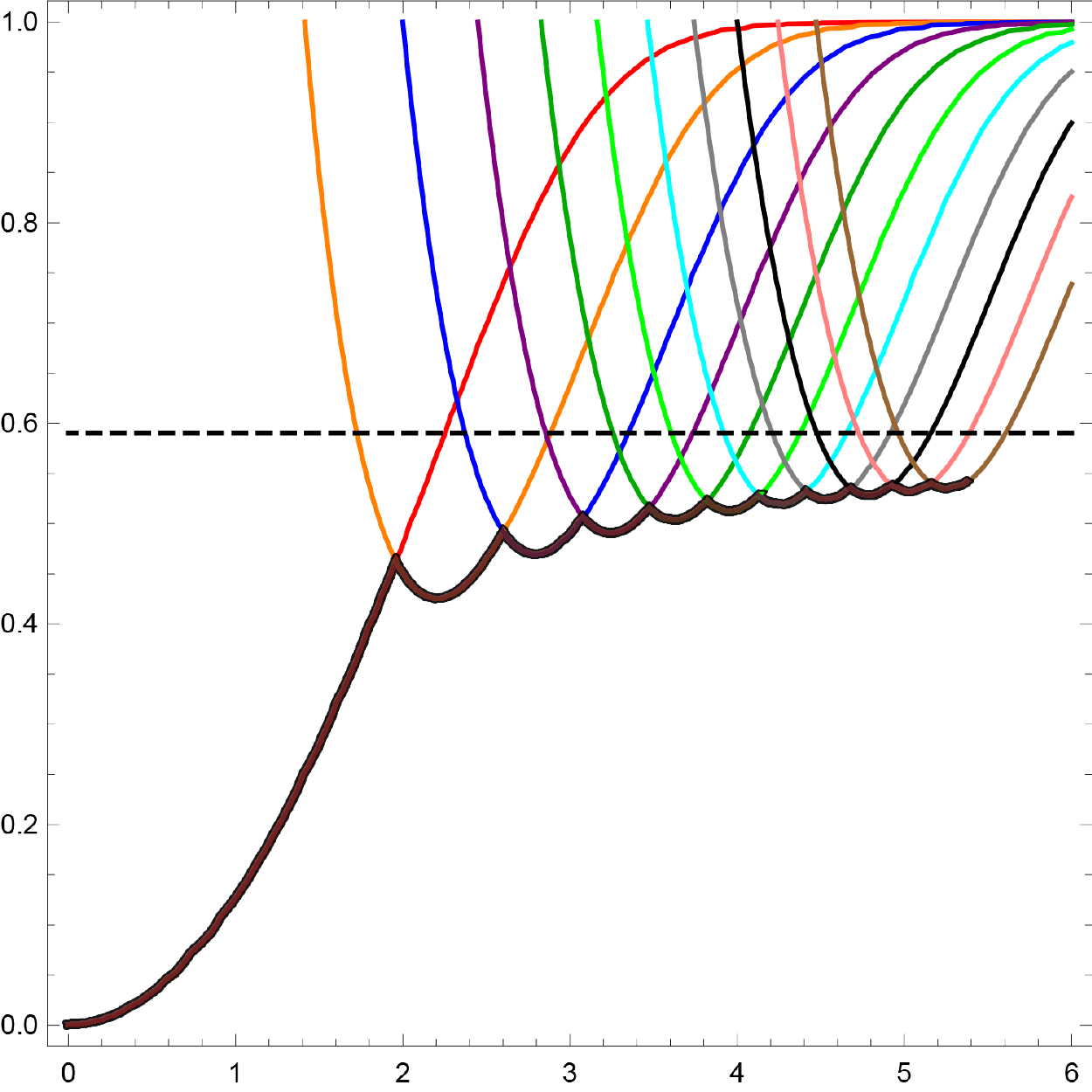}
\caption{Analytic branches $\lambda_1(n,1,R)$, $n=0,...,10$. In particular $n=0$ (red), $n=1$ (orange), $n=2$ (blue), $n=3$ (purple), $n=4$ (dark green), $n=5$ (light green), $n=6$ (cyan), $n=7$ (gray), $n=8$ (black), $n=9$ (pink), $n=10$ (brown). Here the variable on the abscissae axis is the radius $R$. The first eigenvalue is given by the minimum of all analytic branches, and is highlighted in dark brown.}
\label{F1}
\end{figure}

We note that the first eigenvalue has an oscillating behavior as a function of $R$. The black dashed line in Figure \ref{F1} corresponds to $\Theta_0\approx 0.590106$. This suggests that {\bf open problem 2} should have a positive answer. In Figure \ref{F1} we recognize a precise order. In fact, the first eigenvalue is given by $\lambda_1(0,1,R)$ for $R\in(0,R_1)$, then by $\lambda_1(1,1,R)$ for $R\in(R_1,R_2)$, and so on. This was highlighted in the numerical study of \cite{SJ}.

Even though we cannot give precise information on the eigenvalues of disks, explicit computations allow to have more insight on  Theorem \ref{main_averages}. From Theorem \ref{main_averages}  we have that if $|\Omega|\geq\frac{2\pi n}{\beta}$, then
\begin{equation}\label{BBB}
\frac{1}{n}\sum_{j=1}^n\lambda_j(\Omega,\beta)\leq\beta.
\end{equation}
For the disk $B_R$ the condition reads $R\geq\sqrt{\frac{2n}{\beta}}$. We show now that when  $R=\sqrt{\frac{2n}{\beta}}$, the first $n$ eigenvalues are strictly smaller than $\beta$, which clearly implies \eqref{BBB}. Moreover, in this case $\beta$ is an eigenvalue, and, in particular, it is at least the $n+1$-th eigenvalue. 

\begin{prop}
Let $R=\sqrt{\frac{2n}{\beta}}$ for some $n\in\mathbb N$. Then $\lambda_1(B_R,\beta)\leq\cdots\leq\lambda_n(B_R,\beta)<\beta$ and $\beta=\lambda_k(B_R,\beta)$ for some $k\geq n+1$.
\end{prop}
\begin{proof}
It is easy to show that the function (in polar coordinates $(r,t)$)
$$
f_n(r,t)=e^{-\frac{\beta r^2}{4}}r^ne^{int}
$$
is an eigenfunction corresponding to the eigenvalue $\beta$ (see \eqref{laplacian_polar}).  Consider now the functions $f_m(r,t)=e^{-\frac{\beta r^2}{4}}r^me^{imt}$ for $m\in\mathbb N$. 
We prove that
$$
\frac{\int_{B(0,R)}|\nabla^Af_m|^2}{\int_{B(0,R)}|f_m|^2}<\beta
$$
if and only if $m<n$. Since $\{f_m\}_{m=0}^{n-1}$ is an orthogonal family in $L^2(B_R)$, this will imply from the min-max principle \eqref{minmax} that there are at least $n$ eigenvalues strictly below $\beta$, and, as a consequence, that $\beta$ is at least the $n+1$-th eigenvalue. 
A standard computation shows that
$$
\frac{\int_{B(0,R)}|\nabla^Af_m|^2}{\int_{B(0,R)}|f_m|^2}=\beta\frac{m!-m^2\Gamma(m,n)+2m\Gamma(m+1,n)-\Gamma(m+2,n)}{m!-\Gamma(m+1,n)}
$$
where $\Gamma$ denotes here the incomplete Gamma function. Now, for $m=n$ this gives exactly $\beta$. Now, let
$$
G(m,n):=\frac{m!-m^2\Gamma(m,n)+2m\Gamma(m+1,n)-\Gamma(m+2,n)}{m!-\Gamma(m+1,n)}.
$$
We have $G(n,n)=1$. Moreover,
$$
\lim_{n\to+\infty}G(m,n)=1.
$$
Now $G(m,n)<1$ if and only if
$$
-m^2\Gamma(m,n)+(2m+1)\Gamma(m+1,n)-\Gamma(m+2,n)<0
$$
and using the properties of the Gamma function we see that
$$
-m^2\Gamma(m,n)+(2m+1)\Gamma(m+1,n)-\Gamma(m+2,n)=n^me^{-n}(m-n)
$$
which means that $G(m,n)<1$  for all $m<n$. The proof is concluded.
\end{proof}


From Figure \ref{F1} it seems evident that when $R=\sqrt{\frac{2n}{\beta}}$, $\beta$ is exactly the $n+1$-th eigenvalue (recall that $\beta=1$ in Figure \ref{F1}): the analytic branch corresponding to $n\in\mathbb N$ intersects the horizontal line $\lambda=\beta=1$ at $R=\sqrt{2n}$. We are left with the following

\noindent{\bf Open problem 3} Prove that for $R\geq\sqrt{\frac{2n}{\beta}}$ the first $n+1$ eigenvalues of $B_R$ lie below $\beta$;
prove that for any domain with $|\Omega|\geq\frac{2\pi n}{\beta}$  the first $n+1$ eigenvalues lie below $\beta$; improve (if true) inequality \eqref{up_1}  of Theorem \ref{main_averages} as follows: for all $|\Omega|\geq\frac{2\pi n}{\beta}$
$$
\frac{1}{n+1}\sum_{j=1}^{n+1}\lambda_j\leq\beta.
$$

\section{Further examples}\label{sub:examples}

In this Appendix we provide further examples which highlight the difficulties of finding lower bounds for $\lambda_1(\Omega,\beta)$.

We first show that domains with small {\it width} have small first eigenvalue. Recall that the width $\epsilon$ of a domain $\Omega$ is defined as the infimum of the numbers $h>0$ such that, up to isometries, $\Omega$ is contained in a strip $]-\infty,\infty[\times]-h/2,h/2[$.

\begin{thm}\label{small_width}
Let $\Omega$ be a bounded  domain of width $\epsilon$. Then
$$
\lambda_1(\Omega,\beta) \le \frac{\epsilon^2\beta^2}{4}.
$$
\end{thm}
\begin{proof}
Suppose that $\Omega\subset]-\infty,\infty[\times]-\epsilon/2,\epsilon/2[$. We consider the test function $u(x_1,x_2)=e^{i\frac{\beta}{2}x_1x_2}$ in \eqref{minmax}. A standard computation shows that
$$
|\nabla^Au|^2=\beta^2x_2^2.
$$
Since $|u|=1$, we deduce
$$
\lambda_1(\Omega,\beta)\leq\frac{\int_{\Omega}|\nabla^Au|^2}{\int_{\Omega}|u|^2}=\frac{\beta^2\epsilon^2}{4}.
$$
\end{proof}
Of course, this only relevant if the width is small enough.

\medskip

Thanks to Theorem \ref{small_width} we show that there exist convex domains of any measure with small first eigenvalue.

\begin{ex}\label{conv_small}
Let $\Omega_{\epsilon}= ]-L,L[\times ]-\epsilon/2,\epsilon/2[$, $L,\epsilon>0$. For any $L>0$,
\begin{equation}\label{up_ex1}
\lambda_1(\Omega_{\epsilon},\beta)\le\frac{\beta^2\epsilon^2}{4}.
\end{equation}
In particular, $\lambda_1(\Omega_{\epsilon},\beta)\to 0$ as $\epsilon\to 0^+$.
\end{ex}

Note that the upper bound \eqref{up_ex1} does not depend on $L$, so that $\vert \Omega_{\epsilon}\vert$ could be as large (or small) as one wishes. The Rayleigh quotient goes to $0$ proportionally to $\beta^2\epsilon^2$. We recall the Payne-Weinberger inequality \cite{PW} for the first positive eigenvalue of the Neumann Laplacian $\lambda_2^N$:
\begin{equation}\label{PaWe}
\lambda_2^N\geq\frac{\pi^2}{d_{\Omega}^2},
\end{equation}
valid for convex domains. Here $d_{\Omega}$ is the diameter of $\Omega$. Example \ref{conv_small} shows that \eqref{PaWe} does not extend to the first magnetic eigenvalue.

\medskip

One may be tempted to conclude that thin domains, or domains with small area, have small first eigenvalue. To this regard, we remark that topology plays a role: if $\Omega$ is simply connected with small area, it is true that the first eigenvalue is small, see Remark \ref{rem_not_sc}. A bit surprisingly, this is not true if the domain is not simply connected. We show examples of thin domains, of arbitrarily small area and first eigenvalue uniformly bounded away from zero.

\begin{ex}\label{annuli0}
Let $\Gamma$ be a closed simple curve which is the boundary of a smooth bounded domain $\Omega$. Let $\omega_h:=\{x\in\Omega:{\rm dist}(x,\Gamma)<h\}$ be a small tubular neighborhood of $\Gamma$. Then
\begin{enumerate}[1)]
\item if $\frac{\beta|\Omega|}{2\pi}\not\in\mathbb N$,
$$
\lambda_1(\omega_h,\beta)\geq\frac{4\pi^2}{|\Gamma|^2}\min_{n\in\mathbb Z}\left(n-\frac{\beta|\Omega|}{2\pi}\right)^2-\sqrt{\beta}h,
$$
for all $h\in(0,\epsilon_0)$, for some $\epsilon_0>0$ depending on $\Omega$;
\item if  $\frac{\beta|\Omega|}{2\pi}\in\mathbb N$, then
$$
\lambda_1(\omega_h,\beta)\to 0
$$
as $h\to 0^+$
\end{enumerate}
In particular, $\lambda_1(\omega_h,\beta)\to 0$ as $h\to 0^+$ if and only if $\frac{\beta|\Omega|}{2\pi}\in\mathbb N$.
\end{ex}
\begin{proof}
Point 1) follows from Theorem \ref{lower_curve}, while point 2) follows from Proposition \ref{prop0}. We postpone the corresponding proofs at the end of this section.

Alternatively, it is proved in \cite[\S 9]{ru_sch} that
$$
\lambda_j(\omega_h,\beta)\to\lambda_j(\Gamma,\hat A)
$$
as $h\to 0^+$, where $\lambda_j(\Gamma,\hat A)$ denote the magnetic eigenvalues on $\Gamma$ endowed with the restriction of the standard potential $A$. In particular, $\lambda_1(\omega_h,\beta)\to\lambda_1(\Gamma,\hat A)$ as $h\to 0^+$, and  we have seen in Theorem \ref{thm1d} that $\lambda_1(\Gamma,\hat A)=0$ if and only if $\frac{\beta|\Omega|}{2\pi}\in\mathbb N$. 
\end{proof}

 \begin{rem} 1) If the curve $\Gamma=\partial\Omega$ is {\it generic}, in the sense that $\frac{\beta\abs{\Omega}}{2\pi}\notin\mathbb N$, then for small $h$ the domain $\omega_h$ has arbitrarily small area and first eigenvalue bounded away from zero. 

\smallskip

2) If $\Gamma$ is an arbitrary curve then we know from \ref{reverse_curves} that 
$
\lambda_1(\Gamma,\hat A)\leq \frac {\beta}4;
$
 as a consequence, when $h$ is sufficiently small, one has
$
\lambda_1(\omega_h,\beta)<\Theta_0\beta,
$
and we have another family of domains for which $\lambda_1(\Omega,\beta)<\Theta_0\beta$.
\end{rem}

\smallskip

In conclusion we have plenty of  domains with arbitrarily small volume and thickness, and first eigenvalue either arbitrarily close to zero, or bounded away from zero, and this depends on the area enclosed by $\Gamma$. Note that the domains of Example \ref{annuli0} are thin, but not simply connected. If we consider tubes around open curves the first eigenvalue always vanishes as the tube shrinks to the curve.
 
 \begin{ex}\label{open_curve}
 Let $\Gamma$ be an open simple curve and let $
\omega_h:=\{x\in\Omega:{\rm dist}(x,\Gamma)<h\}$.
Then
$$
\lambda_1(\omega_h,\beta)\to 0
$$
as $h\to 0^+$.
\end{ex}
\begin{proof}
Let $\lambda_j(\Gamma,\hat A)$ denote the magnetic eigenvalues on $\Gamma$ endowed with the restriction of $A$, and magnetic Neumann boundary conditions at the endpoints. Since the restriction of $A$ to $\Gamma$ is exact, we conclude that $\lambda_j(\Gamma,\hat A)$ are just the Neumann eigenvalues on $\Gamma$, and in particular, $\lambda_1(\Gamma,\hat A)=0$. Let $\omega_h$ be a tube of size $h$ around $\Gamma$. Then, by \cite[\S 9]{ru_sch} we have that $\lambda_1(\omega_h,\beta)\to \lambda_1(\Gamma,\hat A)=0$ as $h\to 0^+$. Note that the limit $\lambda_1(\omega_h,\beta)\to 0$ follows also by the fact that $\omega_h$ is simply connected and its area goes to zero, see Corollary \ref{variable_1}, point 4).
\end{proof}

We show now a final example of domains with large area, small thickness (i.e., small rolling radius at each point of the boundary), and first eigenvalue close to zero.

\begin{ex}\label{annuli}
Let $A_n:=\{x\in\mathbb R^2:\sqrt{Cn}\leq r\leq\sqrt{Cn+D n^{\alpha}}\}$, where $(r,t)$ are the polar coordinates in $\mathbb R^2$, $C,D>0$ and $0\leq\alpha< \frac{1}{2}$. We have that $|A_n|\to+\infty$ and $\lambda_1(A_n,\beta)\to 0$ as $n\to +\infty$.
\end{ex}
\begin{proof}
 We have $|A_n|=\pi D n^{\alpha}$, while the thickness of the annulus (the difference between the two radii) behaves like $n^{\alpha-1/2}$ as $n\to+\infty$. Taking $C=\frac{2}{\beta}$ and $u(r,t)=e^{int}$ as test functions in \eqref{minmax}, a standard computation shows that $\lambda_1(A_n,\beta)\to 0$ as $n\to+\infty$. With a bit of more work it is possible to deduce the same result for any $C>0$.
 \end{proof} 

These examples show that finding good lower bounds for $\lambda_1(\Omega,\beta)$ is a difficult task. It seems that the condition on the rolling radius of Theorem \ref{thm_lower} is quite natural in many situations.

\medskip

We conclude this section with the proofs of Theorem \ref{lower_curve} and Proposition \ref{prop0} which we have used to discuss Example \ref{annuli0}.

\begin{thm}\label{lower_curve} Let $\Gamma$ be a closed simple curve which is the boundary of a smooth bounded domain $\Omega$ such that $\frac{\beta|\Omega|}{2\pi}\not\in\mathbb N$. Let $\omega_h:=\{x\in\Omega:{\rm dist}(x,\Gamma)<h\}$ be a small tubular neighborhood of $\Gamma$. There exists $\epsilon_0>0$ depending  on $\Omega$ such that, for all $h\in(0,\epsilon_0)$ one has:
$$
\lambda_1(\omega_h,\beta)\geq\lambda_1(\Gamma,\hat A)-\sqrt{\beta} h,
$$
\end{thm}
\begin{proof}
Let $\delta$ denote the injectivity radius of the normal exponential map, which is positive being $\Omega$ smooth. For $r\leq\delta$, let $\Gamma_r=\{x\in\Omega: {\rm dist}(x,\Gamma)=r$) be the equidistant at distance $r$ to the boundary. In Lemma \ref{lem_curve} here below we prove that
\begin{equation}\label{ineq_parallel}
\lambda_1(\Gamma_r,\hat A)\geq\lambda_1(\Gamma,\hat A)-\sqrt{\beta} r,
\end{equation}
for all $r\in[0,\epsilon_0)$, where $\lambda_1(\Gamma_r,\hat A)$ is the lowest eigenvalue of the curve $\Gamma_r$ with potential $\hat A$ (the restriction of the standard potential to the curve). The constant $\epsilon_0$ will be defined in the proof of Lemma \ref{lem_curve}. Inequality \eqref{ineq_parallel} is the main ingredient of the proof of Theorem \ref{lower_curve}.

Take a first eigenfunction $u$ of $\omega_h$ with $\|u\|_{L^2(\omega_h)}=1$. By the coarea formula
$$
\int_{\omega_h}|\nabla^Au|^2=\int_0^h\int_{\Gamma_r}|\nabla^Au|^2dsdr.
$$
Fix a point $p\in\Gamma_r$ and consider an orthonormal frame $(T,N)$ at $p$, where $T$ is tangent to $\Gamma_r$ and $N$ is normal to it. At $p$ we have:
$$
|\nabla^Au|^2=|\langle\nabla^Au,T\rangle|^2+|\langle\nabla^A u,N\rangle|^2\geq |\langle\nabla^Au,T\rangle|^2=|\nabla^{\hat A}u|^2
$$

We can then use the restriction of $u$ as a test-function for the magnetic Laplacian associated to the pair $(\Gamma_r,\hat A)$. This gives, using \eqref{ineq_parallel}
$$
\int_{\Gamma_r}|\nabla^{\hat A}u|^2 ds \geq \lambda_1(\Gamma_r,\hat A)\int_{\Gamma_r}|u|^2\,ds\geq (\lambda_1(\Gamma,\hat A)-\sqrt{\beta}r)\int_{\Gamma_r}|u|^2ds\geq (\lambda_1(\Gamma,\hat A)-\sqrt{\beta}h)\int_{\Gamma_r}|u|^2ds
$$
for all $r\in [0,h]$. Integrating on $[0,h]$ we obtain
$$
\int_{\omega_{h}}|\nabla^Au|^2\geq (\lambda_1(\Gamma,\hat A)-\sqrt{\beta}h)\int_{\omega_h}|u|^2=\lambda_1(\Gamma,\hat A)-\sqrt{\beta}h.
$$
The proof is complete by observing that $\int_{\omega_{h}}|\nabla^Au|^2=\lambda_1(\omega_h,\beta)$.

\smallskip

\end{proof}

\begin{lem}\label{lem_curve} 
Let $\Gamma$ be a closed simple curve which is the boundary of a smooth bounded domain $\Omega$ for which $\frac{\beta|\Omega|}{2\pi}\not\in\mathbb N$. Let $\Gamma_r:=\{x\in\Omega:{\rm dist}(x,\Gamma)=r\}$ be the equidistant at distance $r$ to the boundary. There exists $\epsilon_0>0$ depending  on $\Omega$ such that, for all $r\in(0,\epsilon_0)$ one has:
$$
\lambda_1(\Gamma_r,\hat A)\geq\lambda_1(\Gamma,\hat A)-\sqrt{\beta} r.
$$
\end{lem}
\begin{proof}
Let $\delta$ be the injectivity radius of the normal exponential map; hence the distance function $\rho(x):={\rm dist}(x,\Gamma)$ is smooth for $\rho\leq\delta$. 

Let $\Omega_r=\{x\in\Omega: \rho(x)>r\}$: for $r\in [0,\delta)$, the curve $\Gamma_r=\partial\Omega_r$ is smooth. We set for brevity:
$$
A(r)=\abs{\Omega_r}, \quad L(r)=\abs{\Gamma_r},
$$
so that $A(0)=\abs{\Omega}, L(0)=\abs{\Gamma}$. Theorem \ref{thm1d} gives, for all $r\in [0,\delta)$:

\begin{equation}\label{lowest}
\lambda_1(\Gamma_r,\hat A)=\dfrac{4\pi^2}{L(r)^2}\min_{n\in\mathbb Z}\Big(n-\dfrac{\beta A(r)}{2\pi}\Big)^2.
\end{equation}

We recall the well-known facts that, on the interval $[0,\delta)$,  $A(r)$ is smooth and decreasing, $A'(r)=-L(r)$ and $L'(r)=-2\pi$, so that $A(r)=A(0)-L(0)r+\pi r^2$ (see e.g., \cite[\S 1.2]{tubes}). We will use the inequalities:
\begin{equation}\label{ine}
A(0)-L(0)r\leq A(r)\leq A(0)\quad\text{and}\quad L(r)\leq L(0).
\end{equation}
Since $\frac{\beta|\Omega|}{2\pi}\not\in\mathbb N$, we have from Theorem \ref{thm1d} that $\lambda_1(\Gamma, \hat A)>0$. In particular, there is a unique $n\in\mathbb N$ such that
\begin{equation}\label{cases}
\dfrac{\beta\abs{\Omega}}{2\pi}\in (n-\frac 12,n) \quad\text{or}\quad\dfrac{\beta\abs{\Omega}}{2\pi}\in (n, n+\frac 12].
\end{equation}
Since $\abs{\Omega_r}=A(r)$ is continuous and decreasing, there exists a positive $\epsilon_0\leq\delta$ for which the inequalities in \eqref{cases} continue to hold for the domain $\Omega_r$ for all $r\in [0,\epsilon_0)$, that is:
\begin{equation}\label{casestwo}
\dfrac{\beta A(r)}{2\pi}\in (n-\frac 12,n) \quad\text{or}\quad\dfrac{\beta A(r)}{2\pi}\in (n, n+\frac 12]
\end{equation}
and \eqref{lowest} gives:
\begin{equation}\label{casesthree}
\sqrt{\lambda_1(\Gamma_r,\hat A)}=
\begin{cases}
\dfrac{2\pi n-\beta A(r)}{L(r)}\, & {\rm if}\quad \dfrac{\beta A(r)}{2\pi}\in (n-\frac 12,n),\\
\dfrac{\beta A(r)-2\pi n}{L(r)}\, & {\rm if}\quad  \dfrac{\beta A(r)}{2\pi}\in (n, n+\frac 12]
\end{cases}
\end{equation}
Set for brevity $f(r)=\sqrt{\lambda_1(\Gamma_r,\hat A)}$. We use \eqref{ine} to see that, in the first case of \eqref{casesthree} we have immediately $f(r)\geq f(0)$, while in the second we obtain $f(r)\geq f(0)-r$. Squaring both sides of this last inequality we get
$$
f(r)^2>f(0)^2-2 f(0)r,
$$

From Theorem \ref{reverse_curves} we see that $f(0)=\sqrt{\lambda_1(\Gamma,\hat A)}\leq \frac {\sqrt{\beta}}2$ so that $f(r)^2>f(0)^2-\sqrt{\beta}r$, which is the assertion. 

\end{proof}

\begin{prop}\label{prop0}
Let $\Gamma$ be a closed simple curve which is the boundary of a smooth bounded domain $\Omega$ for which $\frac{\beta|\Omega|}{2\pi}\in\mathbb N$. Let $\omega_h:=\{x\in\Omega:{\rm dist}(x,\Gamma)<h\}$ be a small tubular neighborhood of $\Gamma$. Then
$$
\lim_{h\to 0^+}\lambda_1(\omega_h,\beta)=0.
$$
\end{prop}
\begin{proof}
Let $\delta$ be the injectivity radius of the normal exponential map. From now on we shall assume $h\in(0,\delta)$.

Let $\phi:\Omega\to\mathbb R$ be the solution of
$$
\begin{cases}
\Delta\phi=0\,, & {\rm in\ }\Omega,\\
\langle\nabla u,N\rangle=-\langle A,N\rangle\,, & {\rm on\ }\partial\Omega.
\end{cases}
$$
Define $A'=A+d\phi$, where $A$ is the standard potential. We note that $dA'=dA=\beta$, ${\rm div} A'=0$, and $\langle A',N\rangle=0$ on $\partial\Omega$. The potential $A'$ is often called the {\it Coulomb gauge}. On $\Omega$, and on any $\omega_h$, $A$ and $A'$ differ by an exact one-form, hence $\lambda_1(\omega_h,\beta)=\lambda_1(\omega_h,A')$ for all $h\in(0,\delta)$. 
Consider now the function $v:\Gamma\to\mathbb C$
$$
v(s):=e^{i\int_0^s\hat{A'}},
$$
where $s$ is the arc-length variable on $\Gamma$ and $\hat{A'}$ is the restriction of $A'$ to $\Gamma$. We define a test function $u$ on $\omega_h$ extending $v$ constantly in the normal direction to $\Gamma$. Namely, for $x\in\omega_h$, we set $u(x):=v(s(x))$, where $s(x)$ is the arc-length coordinate of the (unique) nearest point to $x$ on $\Gamma$. By construction $u$ is smooth on $\omega_h$, since $v(|\Gamma|)=e^{i\int_0^s\hat{A'}}=e^{i\beta|\Omega|}=e^{i2\pi n}$ for some $n\in\mathbb N$. Moreover, it does not depend on $h$. Let $p\in\Gamma$, and let $(T,N)$ be an orthonormal frame, where $T$ is tangent to $\Gamma$ at $p$, and $N$ is a unit normal to $\Gamma$ at $p$. Then, at any $p\in\Gamma$ we have $du(T)=iu\hat A'(T)$, so that $d^{A'}u(T)=0$. Moreover $d^{A'}u(N)=0$ since $A'(N)=0$. We conclude that $d^{A'}u=0$ on $\Gamma$, or, equivalently, $\nabla^{A'}u=0$ on $\Gamma$.

This implies that $\|\nabla^{A'}u\|_{L^{\infty}(\omega_h)}\to 0$ as $h\to 0^+$.
From the min-max principle \eqref{minmax} we have
$$
\lambda_1(\omega_h,\beta)=\lambda_1(\omega_h,A')\leq\frac{\int_{\omega_h}|\nabla^{A'}u|^2}{\int_{\omega_h}|u|^2}\leq\frac{\|\nabla^{A'}u\|^2_{L^{\infty}(\omega_h)}|\omega_h|}{|\omega_h|}
$$
since $\int_{\omega_h}|u|^2=\int_{\omega_h}1=|\omega_h|$. The proof is concluded.
\end{proof}


\bibliography{biblioCPSmagnetic1}{}
\bibliographystyle{abbrv}


\end{document}